\documentclass[9pt]{amsart}

\usepackage{amsmath,amssymb,amsthm,amsfonts,mathrsfs}
\usepackage{amscd,stmaryrd}
\usepackage[usenames,dvipsnames]{color}
\usepackage{hyperref}
\usepackage{graphicx,subfigure}
\definecolor{nicegreen}{RGB}{0,180,0}
\hypersetup{colorlinks=true,citecolor=nicegreen,linkcolor=Red,urlcolor=blue,pdfstartview=FitH,
pdfauthor=Koziol,pdftitle=Title}
\usepackage[divide={2.45cm,*,2.45cm}]{geometry}
\usepackage{enumerate}
\usepackage{color}
\usepackage[all]{xy}
\usepackage{tikz-cd}

\allowdisplaybreaks

\newtheorem{thm}{Theorem}[section]
\newtheorem*{thm*}{Theorem}
\newtheorem{cor}[thm]{Corollary}
\newtheorem{lemma}[thm]{Lemma}

\theoremstyle{definition}
\newtheorem{defn}[thm]{Definition}
\newtheorem{rmk}[thm]{Remark}

\newcommand{\qp}{\mathbb{Q}_p}

\newcommand{\T}{\textnormal{T}}
\newcommand{\salpha}{\widehat{s_{\fat{\alpha}}}}
\newcommand{\bo}{\bar{\omega}}
\newcommand{\boldl}{\boldsymbol{\ell}}
\newcommand{\hgt}{\textnormal{ht}}

\newcommand{\Hom}{\textnormal{Hom}}
\newcommand{\End}{\textnormal{End}}
\newcommand{\ind}{\textnormal{Ind}}
\newcommand{\cind}{\textnormal{c-ind}}

\newcommand{\aff}{\textnormal{aff}}
\newcommand{\fat}[1]{\pmb{\boldsymbol{#1}}}

\newcommand*{\longhookrightarrow}{\ensuremath{\lhook\joinrel\relbar\joinrel\rightarrow}}
\newcommand*{\longtwoheadrightarrow}{\ensuremath{\relbar\joinrel\twoheadrightarrow}}
\newcommand{\sub}[2]{\genfrac{}{}{0pt}{}{#1}{#2}}
\newcommand{\sm}[4]{\left(\begin{smallmatrix} #1 & #2 \\ #3 & #4 \end{smallmatrix}\right)}

\newcommand{\cA}{{\mathcal{A}}}

\newcommand{\cC}{{\mathcal{C}}}

\newcommand{\cH}{{\mathcal{H}}}

\newcommand{\cP}{{\mathcal{P}}}

\newcommand{\cX}{{\mathcal{X}}}

\newcommand{\bB}{{\mathbf{B}}}

\newcommand{\bG}{{\mathbf{G}}}
\newcommand{\bH}{{\mathbf{H}}}

\newcommand{\bL}{{\mathbf{L}}}

\newcommand{\bS}{{\mathbf{S}}}
\newcommand{\bT}{{\mathbf{T}}}
\newcommand{\bU}{{\mathbf{U}}}

\newcommand{\bZ}{{\mathbf{Z}}}

\newcommand{\bbF}{{\mathbb{F}}}

\newcommand{\bbN}{{\mathbb{N}}}

\newcommand{\bbQ}{{\mathbb{Q}}}
\newcommand{\bbR}{{\mathbb{R}}}

\newcommand{\bbZ}{{\mathbb{Z}}}

\newcommand{\tG}{\widetilde{G}}

\newcommand{\tI}{\widetilde{I}}

\newcommand{\tW}{\widetilde{W}}

\newcommand{\tZ}{\widetilde{Z}}

\newcommand{\fm}{{\mathfrak{m}}}

\newcommand{\fo}{{\mathfrak{o}}}
\newcommand{\fp}{{\mathfrak{p}}}

\newcommand{\fs}{{\mathfrak{s}}}

\begin{document}
\nocite{}

\title{Hecke module structure on first and top pro-$p$-Iwahori cohomology}
\date{}
\author{Karol Kozio\l}
\address{Department of Mathematics, University of Toronto, 40 St. George Street, Toronto, ON M5S 2E4, Canada} \email{karol@math.toronto.edu}

\begin{abstract}
Let $p\geq 5$ be a prime number, $G$ a split connected reductive group defined over a $p$-adic field, and $I_1$ a choice of pro-$p$-Iwahori subgroup.  Let $C$ be an algebraically closed field of characteristic $p$ and $\mathcal{H}$ the pro-$p$-Iwahori--Hecke algebra over $C$ associated to $I_1$.  In this note, we compute the action of $\mathcal{H}$ on $\textnormal{H}^1(I_1,C)$ and $\textnormal{H}^{\textnormal{top}}(I_1,C)$ when the root system of $G$ is irreducible.  We also give some partial results in the general case.  
\end{abstract}

\maketitle

\section{Introduction}

The mod-$p$ Local Langlands correspondence is now well understood for the group $\textnormal{GL}_2(\qp)$, thanks to the work of Breuil, Colmez, Emerton, Kisin, Pa\v{s}k\={u}nas, and others (see \cite{breuil:icm} and the references therein).  Unfortunately, the smooth mod-$p$ representation theory of $p$-adic reductive groups is still quite mysterious outside of the $\textnormal{GL}_2(\qp)$ case (and other small examples), mainly due to the lack of an explicit description of the so-called supersingular representations.  This constitutes a major hindrance to formulating a mod-$p$ Local Langlands correspondence for general groups.

One potential remedy for these problems is obtained by passing to the derived setting.  Assume $p\geq 5$, let $G$ denote (the group of rational points of) a split connected reductive group over a $p$-adic field $F$, $I_1$ a pro-$p$-Iwahori subgroup of $G$, which we assume to be torsion-free, and $C$ an algebraically closed field of characteristic $p$.  Then Schneider has shown in \cite{schneider:dga} that the (unbounded) derived category of smooth $G$-representations over $C$ is equivalent to the (unbounded) derived category of differential graded $\cH^\bullet$-modules, where $\cH^\bullet$ is the derived Hecke algebra over $C$ with respect to $I_1$.  The equivalence is given by taking a complex $\pi^\bullet$ of smooth $G$-representations to the complex $\textnormal{R}\Gamma(I_1,\pi^\bullet)$, which has a natural action of $\cH^\bullet$.  One hopes that this equivalence will play a role in the burgeoning mod-$p$ Langlands program; see \cite{harris:specs}.

Regrettably, the functor inducing the derived equivalence above is challenging to compute explicitly.  We can try to shed some light on this difficulty by looking at the associated cohomology modules.  Taking $\pi^\bullet$ to be a single $G$-representation $\pi$ concentrated in degree 0, the space $\bigoplus_{i\geq 0}\textnormal{H}^i(I_1,\pi)$ becomes a graded module over the cohomology algebra $\bigoplus_{i\geq 0}\textnormal{H}^i(\cH^\bullet)$.  This cohomology algebra has been studied in the case when $G = \textnormal{SL}_2(F)$ by Ollivier--Schneider in \cite{ollivierschneider:torsion}.  In this note, we focus instead on the cohomology module $\bigoplus_{i\geq 0}\textnormal{H}^i(I_1,\pi)$ in the easiest case when $\pi$ is equal to the trivial representation of $G$, and attempt to understand (part of) the action of $\bigoplus_{i\geq 0}\textnormal{H}^i(\cH^\bullet)$ thereupon.  In a sequel to this note, we will examine the case where $\pi$ is a principal series representation.

We now state more precisely the contents of this note.  Our goal will be to compute the action of the degree 0 part $\textnormal{H}^0(\cH^\bullet)$ of the cohomology algebra, which is naturally isomorphic to the classical pro-$p$-Iwahori--Hecke algebra $\cH := \textnormal{End}_G(\textnormal{c-ind}_{I_1}^G(C))$.  We further restrict our attention to computing the action of $\cH$ on the first and top cohomology spaces, i.e., the spaces $\textnormal{H}^1(I_1,C)$ and $\textnormal{H}^{\textnormal{top}}(I_1,C)$.  After recalling the basic notation and necessary facts about the algebra $\cH$ and its modules in Sections \ref{notation} and \ref{heckealg}, we carry out in Sections \ref{calcofxi} and \ref{frattinisect} some preliminary calculations for the action of $\cH$ on $I_1$-cohomology.  We then apply the results obtained to $\textnormal{H}^1(I_1,C)$ in Section \ref{first} and arrive at the following.  For a precise and more detailed statement, see Theorem \ref{mainthmtrivchar}.

\begin{thm*}[Theorem \ref{mainthmtrivchar}]
Suppose that the root system of $G$ is irreducible, and let $\chi_{\textnormal{triv}}$ denote the trivial character of $\cH$ (described in Subsection \ref{ssingsect} below).
\begin{enumerate}[(a)]
\item  If $F = \qp$ and the root system of $G$ is of type $A_1$, then we have an isomorphism of $\cH$-modules
$$\textnormal{H}^1(I_1,C) \cong \chi_{\textnormal{triv}}^{\oplus a}\oplus \ind_{\cH_T}^{\cH}(\overline{\alpha}),$$
where $\alpha$ is the unique positive root of $G$, and $\ind_{\cH_T}^{\cH}(\overline{\alpha})$ is the simple $\cH$-module given by $\ind_B^G(\overline{\alpha})^{I_1}$.  
\item If $F \neq \qp$ or the root system of $G$ is not of type $A_1$, then we have an isomorphism of $\cH$-modules
$$\textnormal{H}^1(I_1,C) \cong \chi_{\textnormal{triv}}^{\oplus a}\oplus \bigoplus_{\fat{\alpha}\in \Pi_{\aff}/\Omega}\bigoplus_{r = 0}^{f - 1}\fm_{\fat{\alpha},r},$$
where $\fm_{\fat{\alpha},r}$ is a simple supersingular $\cH$-module.  
\end{enumerate}
\end{thm*}

The unexplained notation is as follows: we let $T$ denote a fixed maximal torus of $G$, and $B$ a Borel subgroup containing $T$.  We denote by $\Pi_\aff$ the set of simple affine roots of $G$, and $\Omega$ the subgroup of length 0 elements in the extended affine Weyl group of $G$.  Note that $\Omega$ acts on $\Pi_\aff$.  Further, $a$ is an integer determined by the torus $T$, $\overline{\alpha}$ is the composition of $\alpha$ with the mod-$p$ cyclotomic character, and $f$ is the degree of the residue field of $F$ over $\bbF_p$.

We remark that the conclusion of the theorem above holds without the hypothesis that $I_1$ is torsion-free.  It is also interesting to note that supersingular $\cH$-modules are absent from $\textnormal{H}^1(I_1,C)$ if and only if $F = \qp$ and the root system of $G$ is of type $A_1$.  This mirrors the following expectation (for groups with irreducible root system): there exists an (underived) equivalence between the category of smooth $G$-representations generated by their $I_1$-invariant vectors and the category of $\cH$-modules if and only if $F = \qp$ and the root system of $G$ is of type $A_1$ (cf. \cite{ollivier:foncteur}, \cite{koziol:glnsln}).  When the root system of $G$ is not irreducible, we give some partial qualitative results about $\textnormal{H}^1(I_1,C)$ in Theorem \ref{arbrootsys}.

Finally, we apply the calculations of Section \ref{frattinisect} to compute the top cohomology $\textnormal{H}^{\textnormal{top}}(I_1,C)$, and to show the $\cH$-action is compatible with Poincar\'e duality.

\begin{thm*}[Theorems \ref{htop} and \ref{poincare}]
Suppose that the root system of $G$ is irreducible.  We then have isomorphisms of $\cH$-modules
$$\textnormal{H}^{\textnormal{top}}(I_1,C)\cong \chi_{\textnormal{triv}}$$
and 
$$\textnormal{H}^{i}(I_1,C)^* \cong \textnormal{H}^{\textnormal{top} - i}(I_1,C),$$
where $0 \leq i \leq \textnormal{top}$.  
\end{thm*}

\bigskip

\noindent \textbf{Acknowledgements}.  I would like to thank Florian Herzig for several useful conversations, and the anonymous referee for some helpful suggestions.  During the preparation of this article, funding was provided by NSF grant DMS-1400779 and an EPDI fellowship.

\bigskip

\section{Notation and Preliminaries}\label{notation}

\subsection{Basic Notation}

Let $p\geq 5$ denote a fixed prime number.  Let $F$ be a finite extension of $\qp$, with ring of integers $\fo$, maximal ideal $\fp$, uniformizer $\varpi$ and residue field $k_F$.  We denote by $q = p^f$ the order of the residue field $k_F$.  Given an element $x\in k_F$, we denote by $[x]\in \fo$ its Teichm\"uller lift.  Conversely, if $y\in \fo$, we denote by $\overline{y}\in k_F$ its image in the residue field.

If $\bH$ is an algebraic group over $F$, we denote by $H$ the group $\bH(F)$ of $F$-points.

Let $\bG$ denote a split, connected, reductive group over $F$.  We let $\bZ$ denote the connected center of $G$.  Fix a split maximal torus $\bT$, and let 
$$\langle-,-\rangle: X^*(\bT)\times X_*(\bT)\longrightarrow \bbZ$$ 
denote the natural perfect pairing between characters and cocharacters.  We define a homomorphism $\nu:T\longrightarrow X_*(\bT)$ by the condition
$$\langle\chi,\nu(t)\rangle = -\textnormal{val}(\chi(t))$$
for all $\chi\in X^*(\bT)$ and $t\in T$, and where $\textnormal{val}:F^\times\longrightarrow \bbZ$ is the normalized valuation.  Given $\lambda\in X_*(\bT)$, we have $\nu(\lambda(\varpi^{-1})) = \lambda$.

In the standard apartment corresponding to $\bT$ of the semisimple Bruhat--Tits building of $\bG$, we fix a chamber $\cC$ and a hyperspecial vertex $x_0$ such that $x_0\in \overline{\cC}$.  We let $\cP_{x_0}$ (resp. $I$) denote the parahoric subgroup corresponding to $x_0$ (resp. $\cC$).  We also define $I_1$ to be the pro-$p$-Sylow subgroup of $I$.  The group $I$ (resp. $I_1$) is the \emph{Iwahori subgroup} (resp. \emph{pro-$p$-Iwahori subgroup}) of $G$.  We will later assume that $I_1$ is torsion-free.  We have $\ker(\nu) = T\cap \cP_{x_0} = T \cap I =: T_0$, which is equal to the maximal compact subgroup of $T$.  Furthermore, we set $T_1:= T\cap I_1$; this is the maximal pro-$p$ subgroup of $T_0$.  Then $T_0/T_1$ identifies with the group of $k_F$-points of $\bT$, and we denote this group by $\bT(k_F)$ (note that the assumption that $\bG$ is split ensures $\bT$ has a model over $\fo$).

Denote the root system of $\bG$ with respect to $\bT$ by $\Phi \subset X^*(\bT)$.  We will later restrict to the case where $\Phi$ is irreducible.  We denote by $\Phi^\vee \subset X_*(\bT)$ the coroot system, and given $\alpha\in \Phi$ we denote by $\alpha^\vee\in \Phi^\vee$ the associated coroot.   For a root $\alpha\in \Phi$, we let $\bU_\alpha$ denote the associated root subgroup.  We fix a set of simple roots $\Pi$, and let $\Phi = \Phi^+ \sqcup \Phi^-$ be the decomposition defined by $\Pi$.  We let $\hgt:\Phi \longrightarrow \bbZ$ denote the length function on $\Phi$ with respect to $\Pi$.  Let $\bB$ denote the Borel subgroup corresponding to $\Pi$.  Further, we let $\bU$ denote the unipotent radical of $\bB$, so that $\bU$ is generated by $\bU_\alpha$ for all $\alpha\in \Phi^+$ and $\bB = \bT\ltimes\bU$.

\subsection{Weyl groups}

Let $W_0 := N_{\bG}(\bT)/\bT = N_G(T)/T$ denote the Weyl group of $\bG$.  For $\alpha\in \Phi$, we let $s_\alpha\in W_0$ denote the corresponding reflection (and note that $s_{-\alpha} = s_{\alpha}$).  Then the Coxeter group $W_0$ is generated by $\{s_\alpha\}_{\alpha \in \Pi}$.  We let $\boldl:W_0\longrightarrow \bbN$ denote the length function with respect to the set of simple reflections $\{s_\alpha\}_{\alpha\in \Pi}$.  For $\alpha\in \Phi$, the reflection $s_\alpha$ acts on $\chi\in X^*(\bT)$ (resp. $\lambda\in X_*(\bT)$) by the formula
\begin{eqnarray*}
s_\alpha(\chi) & = & \chi - \langle\chi,\alpha^\vee\rangle\alpha\\
(\textnormal{resp.}~s_{\alpha}(\lambda) & = & \lambda - \langle\alpha,\lambda\rangle\alpha^\vee).
\end{eqnarray*}

Next, we set 
$$\Lambda := T/T_0,\qquad  W := N_G(T)/T_0.$$
Note that the homomorphism $\nu:T\longrightarrow X_*(\bT)$ factors through $\Lambda$, and identifies $\Lambda$ with $X_*(\bT)$.  These groups fit into an exact sequence
$$1\longrightarrow \Lambda \longrightarrow W \longrightarrow W_0 \longrightarrow 1,$$
and the group $W$ acts on the standard apartment $X_*(\bT)/X_*(\bZ)\otimes_{\bbZ}\bbR$ (see \cite[Section I.1]{ss:reptheorysheaves}).  Since $x_0$ is hyperspecial we have $W_0 \cong (N_G(T)\cap \cP_{x_0})/(T\cap \cP_{x_0})$, which gives a section to the surjection $W\longtwoheadrightarrow W_0$.  We will always view $W_0$ as the subgroup of $W$ fixing $x_0$ via this section.  This gives the decomposition
$$W\cong W_0\ltimes \Lambda.$$
The length function $\boldl$ on $W_0$ extends to $W$ (see \cite[Corollary 5.10]{vigneras:hecke1}).

The set of affine roots is defined as $\Phi\times \bbZ$, with the element $(\alpha,\ell)$ taking the value $\alpha(\lambda) + \ell$ on $\lambda\in X_*(\bT)/X_*(\bZ)\otimes_{\bbZ}\bbR$.  We view $\Phi$ as a subset of $\Phi\times \bbZ$ via the identification $\Phi\cong \Phi\times \{0\}$.  \textbf{We assume that $x_0$ and $\cC$ are chosen so that every element of $\Phi$ takes the value $0$ on $x_0$ and every element of $\Phi^+$ is positive on $\cC$.}  We let 
$$\Pi_{\aff}:= \Pi \sqcup \bigsqcup_{i = 1}^d\{(-\alpha_{0}^{(i)}, 1)\}$$ 
denote the set of simple affine roots, where $\{\alpha_{0}^{(i)}\}_{i = 1}^d$ is the set of highest roots of the irreducible components of $\Phi$.

We will use boldface Greek letters to denote generic affine roots.  Given an affine root $\fat{\alpha} = (\alpha,\ell) \in \Phi\times\bbZ$, we let $s_{\fat{\alpha}}\in W$ denote the reflection in the affine hyperplane $\{\lambda\in X_*(\bT)/X_*(\bZ)\otimes_{\bbZ}\bbR: \alpha(\lambda) + \ell = 0\}$.   We define the affine Weyl group $W_\aff\subset W$ to be the subgroup generated by the set $\{s_{\fat{\alpha}}\}_{\fat{\alpha}\in \Pi_\aff}$.  The group $W_\aff$ is a Coxeter group (with respect to the generators $\{s_{\fat{\alpha}}\}_{\fat{\alpha}\in \Pi_\aff}$), and the restriction of $\boldl$ from $W$ to $W_\aff$ agrees with the length function of $W_\aff$ as a Coxeter group.  We also define $\Omega$ as the subgroup of elements of $W$ stabilizing $\cC$; equivalently, $\Omega$ is the subgroup of length 0 elements of $W$.  It is a finitely generated abelian group.  This gives the decomposition
$$W\cong W_\aff\rtimes\Omega.$$ 
The group $W$ acts on the affine roots, and the subgroup $\Omega$ stabilizes the subset $\Pi_\aff$.

We now set 
$$\tW := N_G(T)/T_1.$$
Given any subset $X$ of $W$, we let $\widetilde{X}$ denote its preimage in $\tW$ under the natural projection $\tW\longtwoheadrightarrow W$, so that $\widetilde{X}$ is an extension of $X$ by $\bT(k_F)$.  The length function $\boldl$ on $W$ inflates to $\tW$ via the projection, and similarly the homomorphism $\nu$ on $\Lambda$ inflates to $\widetilde{\Lambda}$.  For typographical reasons we write $\widetilde{X}_\square$ as opposed to $\widetilde{X_\square}$ if the symbol $X$ has some decoration $\square$.  Given some element $w\in W$ we often let $\widehat{w}\in \tW$ denote a specified choice of lift.  Furthermore, given $w\in \tW$, we let $\bar{w}$ denote the image of $w$ in $W_0$ via the projections $\tW\longtwoheadrightarrow W \longtwoheadrightarrow W_0$.

\subsection{Lifts of Weyl group elements}
Let $\bG_{x_0}$ denote the Bruhat--Tits $\fo$-group scheme with generic fiber $\bG$ associated to the point $x_0$ .  Since the group $\bG$ is split, we have a Chevalley system for $\bG$.  In particular, this means that for every $\alpha\in \Phi$, we have a homomorphism of $\fo$-group schemes (cf. \cite[Section 3.2]{bruhattits2})
$$\varphi_{\alpha}:\bS\bL_{2/\fo}\longrightarrow \bG_{x_0}.$$
We normalize $\varphi_\alpha$ as in \cite[Section II.1.3]{jantzen}.  We define $u_\alpha$ to be the homomorphism
\begin{eqnarray*}
u_{\alpha}:\bG_{\textnormal{a}/\fo} & \longrightarrow & \bG_{x_0}\\
x & \longmapsto & \varphi_{\alpha}\left(\begin{pmatrix}1 & x \\ 0 & 1\end{pmatrix}\right).
\end{eqnarray*}

We now define two sets of lifts of affine reflections.  Given a root $\alpha\in \Phi$, define 
$$\fs_{\alpha} := \varphi_\alpha\left(\begin{pmatrix}  0 & 1 \\ -1 & 0\end{pmatrix}\right)\in N_G(T).$$
(These elements are denoted $n_\alpha$ in \cite{springer}.)  We have 
$$\fs_\alpha^2 = \alpha^\vee(-1)\qquad\textnormal{and}\qquad\fs_\alpha^{-1} = \fs_{-\alpha} = \alpha^\vee(-1)\fs_\alpha.$$

On the other hand, given an affine root $\fat{\alpha} = (\alpha,\ell)\in \Phi\times\bbZ$, define
$$\varphi_{\fat{\alpha}}\left(\begin{pmatrix}a & b \\ c & d\end{pmatrix}\right) := \varphi_{\alpha}\left(\begin{pmatrix}a & \varpi^{\ell}b \\ \varpi^{-\ell}c & d\end{pmatrix}\right),$$
and 
$$u_{\fat{\alpha}}(x) := u_\alpha(\varpi^{\ell}x).$$  
Then, for every $\fat{\alpha} = (\alpha,\ell)\in \Pi_\aff$, we define
$$\salpha := \varphi_{\fat{\alpha}}\left(\begin{pmatrix}0 & 1 \\ -1 & 0\end{pmatrix}\right) \in N_G(T).$$
When $\fat{\alpha} = (\alpha,0)\in \Pi_\aff$, we simply write $\widehat{s_\alpha}$.  Note that
$$\salpha =  \varphi_{\alpha}\left(\begin{pmatrix}0 & \varpi^{\ell} \\ -\varpi^{-\ell} & 0\end{pmatrix}\right) = \alpha^\vee(\varpi^\ell)\fs_{\alpha} = \fs_\alpha\alpha^\vee(\varpi^{-\ell}).$$
In particular, if $\alpha\in \Pi$, then $\widehat{s_\alpha} = \fs_\alpha$.

Given $w\in W_0$ with reduced expression $w = s_{\alpha_1}\cdots s_{\alpha_k}$, $\alpha_i\in \Pi$, we define 
\begin{equation}\label{liftredexp}
\widehat{w} := \widehat{s_{\alpha_1}}\cdots\widehat{s_{\alpha_k}} = \fs_{\alpha_1}\cdots \fs_{\alpha_k} \in N_G(T),
\end{equation}
which is a lift of $w$.  By \cite[Propositions 8.3.3 and 9.3.2]{springer}, this expression is well-defined.  Note that if $\alpha\in \Phi\smallsetminus\Pi$, then $\widehat{s_\alpha}$ (computed using \eqref{liftredexp}) is not necessarily equal to $\fs_{\alpha}$.

\subsection{Structure constants}

Let us fix once and for all a total order on $\Phi$.  For two roots $\alpha,\beta\in \Phi$ with $\alpha\neq \pm \beta$, we have
\begin{equation}
\label{comm}
[u_\alpha(x),u_{\beta}(y)] := u_\alpha(x)u_\beta(y)u_{\alpha}(x)^{-1}u_{\beta}(y)^{-1} = \prod_{\sub{i,j > 0}{i\alpha + j\beta\in \Phi}}u_{i\alpha + j\beta}(c_{\alpha,\beta;i,j}x^iy^j),
\end{equation}
where $c_{\alpha,\beta;i,j}$ are some constants, and where the product is taken with respect to the fixed total order $\Phi$.  We also define constants $d_{\alpha,\beta}$ by
$$\fs_\alpha u_{\beta}(x)\fs_\alpha^{-1} = u_{s_\alpha(\beta)}(d_{\alpha,\beta}x).$$
Since the root morphisms $u_\alpha$ come from a Chevalley system and we assume $p\geq 5$, we have $c_{\alpha,\beta;i,j}, d_{\alpha,\beta} \in \fo^\times$ for all choices of $\alpha,\beta\in \Phi$, $i,j > 0$ for which the constants are nonzero.  More precisely, the constants satisfy $c_{\alpha,\beta; i,j}\in \bbZ\cap \fo^{\times}$ and $d_{\alpha,\beta} = \pm 1$ (\cite[Section 3.2]{bruhattits2}).  More generally, if $w\in W_0$ and $\beta\in \Phi$, we define constants $d_{w,\beta}$ by the condition
$$\widehat{w}u_\beta(x)\widehat{w}^{-1} = u_{w(\beta)}(d_{w,\beta}x).$$
Using a reduced expression for $w$, we see that $d_{w,\beta}= \pm 1$ and $d_{w,\beta} = d_{w,-\beta}$ (\cite[Lemma 9.2.2(ii)]{springer}).

\subsection{Miscallany}\subsubsection{} We record a decomposition for future use.  Let $a,d\in \fo^\times, b,c\in \fo$.  Then, in $\textnormal{SL}_2(F)$, we have the following Iwahori decompositions:
\begin{eqnarray*}
\begin{pmatrix}a & b\\ \varpi c & d\end{pmatrix} & = & \begin{pmatrix}1 & bd^{-1} \\ 0 & 1\end{pmatrix}\begin{pmatrix}d^{-1} & 0 \\ 0 & d\end{pmatrix}\begin{pmatrix}1 & 0 \\ \varpi cd^{-1} & 1\end{pmatrix}\\
 & = & \begin{pmatrix} 1 & 0 \\ \varpi ca^{-1} & 1\end{pmatrix}\begin{pmatrix}a & 0 \\ 0 & a^{-1}\end{pmatrix}\begin{pmatrix}1 & ba^{-1} \\ 0 & 1\end{pmatrix}.
\end{eqnarray*}
Applying the homomorphism $\varphi_{\fat{\alpha}}$ for $\fat{\alpha} = (\alpha,\ell)\in \Pi_\aff$ gives
\begin{eqnarray}
\varphi_{\fat{\alpha}}\left(\begin{pmatrix}a & b\\ \varpi c & d\end{pmatrix}\right) & = & u_{\fat{\alpha}}(bd^{-1})\alpha^\vee(d^{-1})u_{-\fat{\alpha}}(\varpi cd^{-1})\label{iwahori1}\\
 & = & u_{-\fat{\alpha}}(\varpi ca^{-1})\alpha^\vee(a)u_{\fat{\alpha}}(ba^{-1}).\label{iwahori2}
\end{eqnarray}

\subsubsection{} We fix for the remainder of the article an algebraically closed field $C$ of characteristic $p$.  This will serve as the field of coefficients for all representations and modules appearing.  We fix once and for all an embedding $k_F\longhookrightarrow C$, and view $k_F$ as a subfield of $C$.  To simplify notation, we will identify the groups of square roots of unity $\mu_2(F), \mu_2(k_F),$ and $\mu_2(C)$.

\subsubsection{} Given $\alpha\in \Phi$, we define a smooth character $\overline{\alpha}:T \longrightarrow k_F^\times \longhookrightarrow C^\times$ by
$$t\longmapsto \overline{\varpi^{-\textnormal{val}(\alpha(t))}\alpha(t)}.$$
In particular, the character $\overline{\alpha}$ restricted to $T_0$ is given by
$$t_0\longmapsto\overline{\alpha(t_0)},$$
and therefore does not depend on the choice of uniformizer.

\subsubsection{} Finally, if $A$ is some set, $a,b\in A$, and $A'\subset A$, we define
\begin{align*}
\mathbf{1}_{A'}(a)  :=  & \begin{cases}1 & \textnormal{if}~a\in A',\\ 0 & \textnormal{if}~a\not\in A',\end{cases}\\
\delta_{a,b} := \mathbf{1}_{\{b\}}(a)  = & \begin{cases}1 & \textnormal{if}~a = b,\\ 0 & \textnormal{if}~a\neq b.\end{cases}
\end{align*}

\bigskip

\section{Hecke algebras}\label{heckealg}

We review some facts related to pro-$p$-Iwahori--Hecke algebras.

\subsection{Definitions}\label{ssectdefs}  We let $\cH$ denote the pro-$p$-Iwahori--Hecke algebra of $G$ with respect to $I_1$ over $C$:
$$\cH:= \End_{G}\big(\cind_{I_1}^G(C)\big),$$
where $C$ denotes the trivial $I_1$-module over $C$ (see \cite{vigneras:hecke1} for details).  Using the adjunction isomorphism
$$\cH\cong \Hom_{I_1}(C,\cind_{I_1}^G(C)) = \cind_{I_1}^G(C)^{I_1},$$
we view $\cH$ as the convolution algebra of compactly supported, $C$-valued, $I_1$-biinvariant functions on $G$.  Given $g\in G$, we let $\T_g$ denote characteristic function of $I_1gI_1$. If $w\in \tW$ and $\dot{w}\in N_G(T)$ is a lift of $w$, then the double coset $I_1\dot{w}I_1$ does not depend on the choice of lift, and we will therefore write $\T_w$ for $\T_{\dot{w}}$.  The set $\{\T_w\}_{w\in \tW}$ then gives a basis of $\cH$.

The elements of $\cH$ satisfy braid relations and quadratic relations (see \cite{vigneras:hecke1}).  We will not need their precise definition here, but we only remark that they imply that $\cH$ is generated (as an algebra) by $\{\T_{\salpha}\}_{\fat{\alpha}\in \Pi_\aff}$ and $\{\T_{\omega}\}_{\omega\in \widetilde{\Omega}}$.

Finally, we define $\cH_\aff$ to be the $C$-vector subspace of $\cH$ which is spanned by $\{\T_w\}_{w\in \tW_\aff}$.  By the braid and quadratic relations, this forms a subalgebra of $\cH$, called the \emph{affine pro-$p$-Iwahori--Hecke algebra}.  The algebra $\cH_\aff$ is generated by $\{\T_{\salpha}\}_{\fat{\alpha}\in \Pi_\aff}$ and $\{\T_{t_0}\}_{t_0\in \bT(k_F)}$.  If $\Phi = \Phi^{(1)}\sqcup \ldots \sqcup \Phi^{(d)}$ is the decomposition of $\Phi$ into irreducible components, we get a corresponding decomposition of simple affine roots $\Pi_\aff = \Pi_\aff^{(1)}\sqcup \ldots \sqcup \Pi_\aff^{(d)}$.  For $1\leq i \leq d$, the subalgebras of $\cH_\aff$ generated by $\{\T_{\salpha}\}_{\fat{\alpha}\in \Pi_\aff^{(i)}}$ and $\{\T_{t_0}\}_{t\in \bT(k_F)}$ are called the \emph{irreducible components of $\cH_\aff$}.

\subsection{Supersingular $\cH$-modules}\label{ssingsect}
We first recall two characters of $\cH$.  

\begin{defn}
\begin{enumerate}[(a)]
\item The \emph{trivial and sign characters} $\chi_{\textnormal{triv}},\chi_{\textnormal{sign}}:\cH\longrightarrow C$ are defined by
$$\chi_{\textnormal{triv}}(\T_{w}) = q^{\boldl(w)} \qquad \textnormal{and}\qquad \chi_{\textnormal{sign}}(\T_{w}) = (-1)^{\boldl(w)},$$
respectively (using the convention that $0^0 = 1$).  
\item Given any subalgebra $\cA$ of $\cH$ (e.g., $\cH_\aff$, etc.), we define the \emph{trivial and sign characters of $\cA$} to be the restrictions of $\chi_{\textnormal{triv}}$ and $\chi_{\textnormal{sign}}$ to $\cA$.   
\item Given a character $\xi:\bT(k_F) \longrightarrow C^\times$ which satisfies $\xi\circ\alpha^\vee([x]) = 1$ for every $\alpha\in \Pi$ and every $x\in k_F^\times$, we define the \emph{twists of $\chi_{\textnormal{triv}}$ and $\chi_{\textnormal{sign}}$ by $\xi$} to be the characters of $\cH_\aff$ defined by
$$\begin{cases}(\xi\otimes\chi_{\textnormal{triv}})(\T_{t_0}) = \xi(t_0)^{-1} & \textnormal{if}~t_0\in \bT(k_F), \\ (\xi\otimes\chi_{\textnormal{triv}})(\T_{\salpha}) = 0 & \textnormal{if}~\fat{\alpha}\in \Pi_\aff,\end{cases}\qquad\textnormal{and}\qquad\begin{cases}(\xi\otimes\chi_{\textnormal{sign}})(\T_{t_0}) = \xi(t_0)^{-1} & \textnormal{if}~t_0\in \bT(k_F), \\ (\xi\otimes\chi_{\textnormal{sign}})(\T_{\salpha}) = -1 & \textnormal{if}~\fat{\alpha}\in \Pi_\aff.\end{cases}$$
\end{enumerate}
\end{defn}

We now move on to the notion of supersingularity.  We will not give the precise definition here, but simply content ourselves with the following equivalent formulation (cf. \cite[Corollary 6.13(2) and Theorem 6.15]{vigneras:hecke3}).  Recall that a supersingular character of $\cH_\aff$ is one whose restriction to each irreducible component of $\cH_\aff$ is different from a twist of the trivial or sign character.  Then a simple right $\cH$-module $\fm$ is supersingular if and only if the restriction of $\fm$ to $\cH_\aff$ contains a supersingular character.

\subsection{Hecke action on cohomology}\label{heckecoh}

We now discuss Hecke actions on (continuous) group cohomology.  For a reference, see \cite{lee:heckeaction}, \cite{kugaparrysah}, and \cite{rhiewhaples}.

Let $g\in G$, and let us write
\begin{equation}\label{xiforcoh}
I_1gI_1 = \bigsqcup_{x\in \cX} I_1g_x
\end{equation}
where $\cX$ is some finite index set.  Given $h\in I_1$, we have $g_xh\in I_1gI_1$, so we may write
$$g_xh = \xi_x(h)g_{x(h)}$$
for some $\xi_x(h)\in I_1$ and $x(h)\in \cX$ (we suppress the dependence on $g$ from notation).  If $h, h'\in I_1$, we have the cocycle conditions
\begin{equation}\label{cocycle}
x(hh') = x(h)(h'),\qquad\textnormal{and}\qquad \xi_x(hh') = \xi_x(h)\xi_{x(h)}(h').
\end{equation}

Suppose now that we have a smooth $G$-representation $\pi$ and an inhomogeneous cocycle $f\in \textnormal{Z}^i(I_1,\pi)$.  We define another cocycle $f\cdot\T_g$ by
\begin{equation}\label{actoncocyc}
(f\cdot\T_g)(h_1,\ldots, h_i) = \sum_{x\in \cX}g_x^{-1}.f\left(\xi_x(h_1),~ \xi_{x(h_1)}(h_2),~ \xi_{x(h_1h_2)}(h_3),~ \ldots~,~ \xi_{x(h_1\cdots h_{i - 1})}(h_i)\right).
\end{equation}
In particular, if $g$ normalizes $I_1$, this reduces to
$$(f\cdot\T_g)(h_1,\ldots, h_i) = g^{-1}.f\left(gh_1g^{-1},~ gh_2g^{-1},~ \ldots~,~ gh_ig^{-1}\right).$$
By passing to cohomology, equation \eqref{actoncocyc} gives an well-defined action of $\cH$ on $\textnormal{H}^i(I_1,\pi)$.

More functorially, the action of the operator $\T_g$ on $\textnormal{H}^i(I_1,\pi)$ is given by the composition
\begin{equation}\label{heckefunctorial}
\begin{tikzcd}
\quad\textnormal{H}^i(I_1,\pi)\quad \ar[r, "\textnormal{res}^{I_1}_{I_1\cap gI_1g^{-1}}"] &  \quad\textnormal{H}^i(I_1\cap gI_1g^{-1},\pi)\quad \ar[r, "g^{-1}_*"] & \quad\textnormal{H}^i(I_1\cap g^{-1}I_1g,\pi)\quad \ar[r, "\textnormal{cor}^{I_1}_{I_1\cap g^{-1}I_1g}"] & \quad\textnormal{H}^i(I_1,\pi),\quad
\end{tikzcd}
\end{equation}
where the first map is the restriction from $I_1$ to $I_1\cap gI_1g^{-1}$, the second map is the conjugation by $g^{-1}$, and the third map is the corestriction from $I_1\cap g^{-1}I_1g$ to $I_1$.

In this note, we will be concerned with the case where $\pi = C$, the trivial one-dimensional representation of $G$ over $C$.  It follows easily from the definitions that $\textnormal{H}^0(I_1,C) \cong \chi_{\textnormal{triv}}$; we will examine the higher cohomology groups in the sections below.

\bigskip

\section{Calculations for Hecke action}\label{calcofxi}

In this section we calculate the maps $\xi_x(h)$ and $x(h)$ which appear in the action of $\cH$ on cohomology.  Since $\cH$ is generated by $\{\T_{\salpha}\}_{\fat{\alpha}\in \Pi_\aff}$ and $\{\T_\omega\}_{\omega\in \widetilde{\Omega}}$, and the action of $\T_{\omega}$ is given by the simplified version of equation \eqref{actoncocyc}, it suffices to understand $\xi_x(h)$ and $x(h)$ when $g = \salpha$.

Taking $g = \salpha$ in the decomposition \eqref{xiforcoh} gives
$$I_1\widehat{s_{\fat{\alpha}}}I_1 = \bigsqcup_{x\in k_F}I_1\widehat{s_{\fat{\alpha}}}u_{\fat{\alpha}}([x]).$$
Therefore, our index set $\cX$ is $k_F$, and for $x\in k_F$, we have $g_x = \salpha u_{\fat{\alpha}}([x])$.  Since $I_1$ is generated by $T_1$ and $u_{\fat{\beta}}(\fo)$ for all $\fat{\beta} \in \Phi^+\times\{0\} \sqcup \Phi^-\times\{1\}$, by equations \eqref{cocycle} it suffices to understand $\xi_x(t), \xi_x(u_{\fat{\beta}}(y)), x(t)$, and $x(u_{\fat{\beta}}(y))$, where $t\in T_1, \fat{\beta}\in \Phi^+\times\{0\} \sqcup \Phi^-\times\{1\}$ and $y\in \fo$.

\begin{lemma}
\label{conj}
Let $\fat{\alpha} = (\alpha,\ell)\in  \Pi_\aff$, $\fat{\beta} = (\beta,m)\in \Phi^+\times\{0\}\sqcup \Phi^-\times\{1\}$ and $x\in k_F, y\in \fo$.  We have:
\begin{enumerate}[(a)]
\item if $\fat{\alpha} = \fat{\beta}$, then
\begin{eqnarray*}
x(u_{\fat{\alpha}}(y)) & = & x + \overline{y},\\
\xi_x(u_{\fat{\alpha}}(y)) & = & u_{-\fat{\alpha}}([x + \overline{y}] - [x] - y);
\end{eqnarray*}\label{conja}
\item if $\fat{\alpha} \neq \fat{\beta}$ and $\alpha\neq -\beta$, then
\begin{eqnarray*}
x(u_{\fat{\beta}}(y)) & = & x,\\
\xi_x(u_{\fat{\beta}}(y)) & = & \left(\prod_{\sub{i,j > 0}{i\alpha + j\beta\in \Phi}}u_{s_\alpha(i\alpha + j\beta)}\left(d_{\alpha,i\alpha + j\beta}c_{\alpha,\beta;i,j}\varpi^{-i\ell + jm - j\ell\langle \beta,\alpha^\vee\rangle}[x]^iy^j\right)\right) \cdot ~u_{s_\alpha(\beta)}\left(d_{\alpha,\beta}\varpi^{m - \ell\langle\beta,\alpha^\vee\rangle}y\right);
\end{eqnarray*}\label{conjb}
\item if $\fat{\alpha}\neq\fat{\beta}$ and $\alpha = -\beta$, then
\begin{eqnarray*}
x(u_{\fat{\beta}}(y)) & = & x,\\
\xi_x(u_{\fat{\beta}}(y)) & = & u_{\fat{\alpha}}(-\varpi y\nu^{-1})\alpha^\vee(\nu^{-1})u_{\fat{\beta}}([x]^2y\nu^{-1})\\
 & = & u_{\fat{\beta}}([x]^2y\nu'^{-1})\alpha^\vee(\nu')u_{\fat{\alpha}}(-\varpi y\nu'^{-1}),
\end{eqnarray*}
where $\nu := 1 + \varpi[x]y$ and $\nu' := 1 - \varpi[x]y$; \label{conjc}
\item for $t\in T_1$, we have
\begin{eqnarray*}
x(t) & = & x,\\ 
\xi_x(t) & = & t^{s_\alpha}u_{-\fat{\alpha}}\left((1 - \alpha(t)^{-1})[x]\right),
\end{eqnarray*}
where $t^{s_\alpha} := \salpha t\salpha^{-1}$.  \label{conjd}
\end{enumerate}
\end{lemma}

\begin{proof}
\begin{enumerate}[(a)]
\item We have 
\begin{eqnarray*}
\salpha u_{\fat{\alpha}}([x] + y) & = & \varphi_{\fat{\alpha}}\left(\begin{pmatrix} 0 & 1 \\ -1 & 0\end{pmatrix}\begin{pmatrix}1 & [x] + y \\ 0 & 1\end{pmatrix}\right)\\
& = & \varphi_{\fat{\alpha}}\left(\begin{pmatrix}1 & 0 \\ [x + \overline{y}] - [x] - y & 1\end{pmatrix}\begin{pmatrix} 0 & 1 \\ -1 & 0\end{pmatrix}\begin{pmatrix}1 & [x + \overline{y}] \\ 0 & 1\end{pmatrix}\right)\\
& = & u_{-\fat{\alpha}}([x + \overline{y}] - [x] - y)\salpha u_{\fat{\alpha}}([x + \overline{y}]).
\end{eqnarray*}

\item We have
\begin{eqnarray*}
\salpha u_{\fat{\alpha}}([x])u_{\fat{\beta}}(y) & = & \fs_\alpha \alpha^\vee(\varpi)^{-\ell}u_{\alpha}(\varpi^\ell[x])u_\beta(\varpi^my)\\
 & \stackrel{\eqref{comm}}{=} & \fs_\alpha \alpha^\vee(\varpi)^{-\ell} \left(\prod_{\sub{i,j > 0}{i\alpha + j\beta\in \Phi}}u_{i\alpha + j\beta}\left(c_{\alpha,\beta;i,j}\varpi^{i\ell + jm}[x]^iy^j\right)\right)u_\beta(\varpi^my)u_{\alpha}(\varpi^\ell[x])\\
  & = & \fs_\alpha \left(\prod_{\sub{i,j > 0}{i\alpha + j\beta\in \Phi}}u_{i\alpha + j\beta}\left(c_{\alpha,\beta;i,j}\varpi^{-i\ell + jm - j\ell\langle \beta,\alpha^\vee\rangle}[x]^iy^j\right)\right) \\
  & & \cdot ~u_\beta\left(\varpi^{m - \ell\langle\beta,\alpha^\vee\rangle}y\right)\alpha^\vee(\varpi)^{-\ell}u_{\alpha}(\varpi^\ell[x])\\
  & = & \left(\prod_{\sub{i,j > 0}{i\alpha + j\beta\in \Phi}}u_{s_\alpha(i\alpha + j\beta)}\left(d_{\alpha,i\alpha + j\beta}c_{\alpha,\beta;i,j}\varpi^{-i\ell + jm - j\ell\langle \beta,\alpha^\vee\rangle}[x]^iy^j\right)\right) \\
  & & \cdot ~u_{s_\alpha(\beta)}\left(d_{\alpha,\beta}\varpi^{m - \ell\langle\beta,\alpha^\vee\rangle}y\right)\salpha u_{\fat{\alpha}}([x]).
\end{eqnarray*}

\item By assumption, we may write $\fat{\beta} = (-\alpha,1 - \ell)$.  Set $\nu := 1 + \varpi[x]y$, $\nu':= 1 - \varpi[x]y$.  We have 
\begin{eqnarray*}
\salpha u_{\fat{\alpha}}([x])u_{\fat{\beta}}(y) & = &  \varphi_{\fat{\alpha}}\left(\begin{pmatrix} 0 & 1\\ -1 & 0 \end{pmatrix}\begin{pmatrix} 1 & [x] \\ 0 & 1 \end{pmatrix}\begin{pmatrix} 1 & 0 \\ \varpi y & 1 \end{pmatrix}\right)\\
 & = & \varphi_{\fat{\alpha}}\left(\begin{pmatrix}\nu' & -\varpi y\\ \varpi[x]^2y & \nu\end{pmatrix}\begin{pmatrix} 0 & 1\\ -1 & 0 \end{pmatrix}\begin{pmatrix} 1 & [x] \\ 0 & 1 \end{pmatrix}\right)\\
 & \stackrel{\eqref{iwahori1}}{=} & u_{\fat{\alpha}}(-\varpi y\nu^{-1})\alpha^\vee(\nu^{-1})u_{-\fat{\alpha}}(\varpi[x]^2y\nu^{-1})\salpha u_{\fat{\alpha}}([x])\\
 & \stackrel{\eqref{iwahori2}}{=} & u_{-\fat{\alpha}}(\varpi[x]^2y\nu'^{-1})\alpha^\vee(\nu')u_{\fat{\alpha}}(-\varpi y\nu'^{-1})\salpha u_{\fat{\alpha}}([x]).
\end{eqnarray*}

\item We have
\begin{eqnarray*}
\salpha u_{\fat{\alpha}}([x])t & = &  t^{s_\alpha}\salpha u_{\fat{\alpha}}(\alpha(t)^{-1}x)\\
 & = & t^{s_\alpha}\salpha u_{\fat{\alpha}}\left((\alpha(t)^{-1} - 1)[x]\right)u_{\fat{\alpha}}([x])\\
 & = & t^{s_\alpha} u_{-\fat{\alpha}}\left((1 - \alpha(t)^{-1})[x]\right)\salpha u_{\fat{\alpha}}([x]).
\end{eqnarray*}
\end{enumerate}
\end{proof}

\bigskip

\section{Calculation of Frattini quotients}\label{frattinisect}

In this section we calculate Frattini quotients of some subgroups appearing in the sequel.  For a profinite group $J$, we let $[J,J]$ denote the closed subgroup generated by all commutators, and let $J^{\textnormal{ab}} = J/[J,J]$ denote the abelianization.  Recall that we may write the pro-$p$-Iwahori subgroup $I_1$ in terms of generators as
$$I_1 = \left\langle T_1, u_{\alpha}(\fo),u_{\beta}(\fp): \alpha\in \Phi^+,\beta\in \Phi^-\right\rangle.$$

\begin{lemma}\label{i1ab}
We have
$$I_1^{\textnormal{ab}} = \overline{T}_1 \oplus \bigoplus_{\fat{\alpha}\in \Pi_{\aff}}u_{\fat{\alpha}}(\fo/\fp),$$
where $\overline{T}_1 := T_1/\langle\alpha^\vee(1 + \fp)\rangle_{\alpha\in\Phi}$.  
\end{lemma}

\begin{rmk}
See also \cite[Section 3.8]{ollivierschneider:torsion} and \cite{cornutray}.
\end{rmk}

\begin{proof}
By working one irreducible component of $\Phi$ at a time, it suffices to prove the lemma when $\Phi$ is irreducible.  Let $\alpha_0$ denote the highest root of $\Phi$, and define $J$ as the closed subgroup of $I_1$ given by
$$\left\langle\gamma^\vee(1 + \fp),~ u_\alpha(\fo),~ u_{\beta}(\fp),~ u_{-\alpha_0}(\fp^2):\begin{array}{l}\diamond~ \gamma\in \Phi,\\ \diamond~ 1 < \hgt(\alpha)\leq \hgt(\alpha_0)\\ \diamond~ \hgt(-\alpha_0) < \hgt(\beta) \leq -1~\textnormal{or}~\hgt(\beta) = 1. \end{array}\right\rangle$$
We will prove that $J = [I_1,I_1]$.  

\begin{enumerate}[(a)]
\item Take $\alpha \in \Phi^+, y\in \fo$ and $t\in T_1$.  We have
$$[t,u_{\alpha}(y)] = u_{\alpha}((\alpha(t) - 1)y)\in u_{\alpha}(\fp).$$
Since $p\geq 5$, taking $t = \alpha^\vee(x)$ for $x\in 1 + \fp$, $x\neq 1$, gives $\alpha(t) - 1 = x^2 - 1\in \fp\smallsetminus\fp^2$.  Thus, we obtain $u_{\alpha}(\fp)\subset [I_1,I_1]$.  Similarly, taking $\alpha\in \Phi^-$ and $y\in \fp$, we have $u_{\alpha}(\fp^2)\subset [I_1,I_1]$.

\item  If $\hgt(\alpha_0) > 1$ (that is, if $\Phi$ is not of type $A_1$), we may choose $\alpha,\beta\in \Phi^+$ such that $\alpha + \beta = \alpha_{0}$.  Taking $x,y\in \fo$, equation \eqref{comm} now gives
$$[u_\alpha(x),u_{\beta}(y)] = u_{\alpha_{0}}(c_{\alpha,\beta;1,1}xy)$$
which implies that $u_{\alpha_{0}}(\fo)\subset [I_1,I_1]$ since $c_{\alpha,\beta;1,1}\in \fo^\times$.  Using equation \eqref{comm} and descending induction on height we obtain $u_\alpha(\fo)\subset [I_1,I_1]$ for all $\alpha\in \Phi^+$ with $1 < \hgt(\alpha) \leq \hgt(\alpha_{0})$.

Assume again that the root system $\Phi$ is not of type $A_1$, and suppose $\hgt(\alpha) = -1$.  By the assumption on the root system, there exists $\beta\in \Pi$ such that $\alpha - \beta\in \Phi^-$.  Taking $x\in \fp,y\in\fo$, we get
$$[u_{\alpha - \beta}(x),u_{\beta}(y)] = \prod_{\sub{i,j > 0}{i(\alpha - \beta) + j\beta\in \Phi}}u_{i(\alpha - \beta) + j\beta}(c_{\alpha - \beta,\beta;i,j}x^iy^j).$$
When $(i,j) = (1,1)$, we get the term $u_{\alpha}(c_{\alpha - \beta,\beta;1,1}xy)\in u_{\alpha}(\fp)$.  If $(i,j)$ is any other pair of integers appearing in the product, we must have $i \geq 2$, and therefore $u_{i(\alpha - \beta) + j\beta}(c_{\alpha - \beta,\beta;i,j}x^iy^j)\in u_{i(\alpha - \beta) + j\beta}(\fp^2)\subset [I_1,I_1]$ (by part (a)).  We conclude that $u_\alpha(\fp)\subset [I_1,I_1]$.  Proceeding again by descending induction and using the same considerations as above, we obtain $u_{\alpha}(\fp)\subset [I_1,I_1]$ for all $\alpha\in \Phi^-$ with $\hgt(-\alpha_0) < \hgt(\alpha) \leq - 1$.

\item  We have, for $x\in \fp, y\in \fo$ and $\alpha\in \Phi^+$, 
\begin{eqnarray*}
[u_{-\alpha}(x), u_{\alpha}(y)] & = & \varphi_{\alpha}\left(\begin{pmatrix}1 - xy & xy^2\\ - x^2y & (1 + xy)^2 - xy\end{pmatrix}\right)\\
 & \stackrel{\eqref{iwahori2}}{=} & u_{-\alpha}(-x^2y(1 - xy)^{-1})\alpha^\vee(1 - xy)u_{\alpha}(xy^2(1 - xy)^{-1})\\
 & \in & [I_1,I_1]\alpha^\vee(1 - xy)[I_1,I_1].
\end{eqnarray*}
Therefore, we see that $\alpha^{\vee}(1 + \fp)\subset [I_1,I_1]$.  
\end{enumerate}

Points (a)--(c) imply that $J\subset [I_1,I_1]$, and that $J$ is normal in $I_1$.  Moreover, the commutator formula \eqref{comm} shows that $I_1/J$ is abelian, and equal to 
$$\overline{T}_1 \oplus \bigoplus_{\fat{\alpha}\in \Pi_{\aff}}u_{\fat{\alpha}}(\fo/\fp).$$
This gives $[I_1,I_1]\subset J$.  
\end{proof}

\begin{cor}
The Frattini quotient of $I_1$ is isomorphic to
$$\overline{T}_1/\overline{T}_1^p\oplus \bigoplus_{\fat{\alpha}\in \Pi_{\aff}}u_{\fat{\alpha}}(\fo/\fp).$$
\end{cor}

For $\fat{\alpha} = (\alpha,\ell)\in \Pi_\aff$ and $0\leq r \leq f - 1$, we define $\eta_{\fat{\alpha},r}:I_1\longrightarrow C$ to be the homomorphism
$$\eta_{\fat{\alpha},r}\left(t\prod_{\beta\in \Phi^-}u_{\beta}(\varpi x_{\beta})\prod_{\gamma\in \Phi^+}u_{\gamma}(x_\gamma)\right) = \overline{x_{\alpha}}^{p^r},$$
where $t\in T_1, x_\beta,x_\gamma\in \fo$.

\begin{cor}\label{h1basis}
As vector spaces, we have
$$\textnormal{H}^1(I_1,C) = \Hom^{\textnormal{cts}}(I_1,C) = \Hom^{\textnormal{cts}}(\overline{T}_1,C)\oplus\bigoplus_{\fat{\alpha}\in \Pi_\aff} \bigoplus_{r = 0}^{f - 1}C\eta_{\fat{\alpha},r}.$$
\end{cor}

We now consider deeper congruence subgroups of $I_1$.  For $m\geq 1$, let $I_m$ denote the open compact subgroup 
$$I_m := \left\langle \lambda(1 + \fp^m), u_\alpha(\fp^{m - 1}), u_\beta(\fp^m): \lambda\in X_*(\bT), \alpha\in \Phi^+, \beta\in \Phi^-\right\rangle.$$
By \cite[Section I.2]{ss:reptheorysheaves}, $I_m$ is normal in $I_1$.

\begin{lemma}\label{frattinilemma}
Let $e := [F:\qp]/f$ denote the absolute ramification index of $F$.  If $m > e$, then the Frattini quotient of $I_m$ is equal to 
$$I_m/I_{m + e} = T_{m}/T_{m + e}\oplus\bigoplus_{\alpha\in \Phi^+}u_{\alpha}(\fp^{m - 1}/\fp^{m + e - 1})\oplus\bigoplus_{\beta\in \Phi^-}u_{\beta}(\fp^m/\fp^{m + e}),$$
where $T_{n} := \langle\lambda(1 + \fp^n):\lambda\in X_*(\bT)\rangle$.  
\end{lemma}

\begin{proof}
We show that the subgroup $I_m^p[I_m,I_m]$ is equal to $I_{m + e}$.  First, by \cite[XIV, Proposition 9]{serre:localfields}, and the definition of ramification index, we have $(1 + \fp^m)^p = 1 + \fp^{m + e}$ and $p\fp^n = \fp^{n + e}$ for $n \geq 0$.  
Hence, for $\lambda\in X_*(\bT), \alpha\in \Phi^+$ and $\beta\in \Phi^-$, we have
\begin{equation}\label{frattini}
\begin{aligned}
\lambda(1 + \fp^{m + e})  = \lambda\left((1 + \fp^m)^p\right) = \lambda(1 + \fp^m)^p &\in I_m^p[I_m,I_m],\\
u_{\alpha}(\fp^{m + e - 1})  = u_{\alpha}(p\fp^{m - 1}) = u_\alpha(\fp^{m - 1})^p & \in I_m^p[I_m,I_m],\\
u_{\beta}(\fp^{m + e})  = u_{\beta}(p\fp^m) = u_\beta(\fp^m)^p & \in I_m^p[I_m,I_m].
\end{aligned}
\end{equation}
This gives $I_{m + e}\subset I_m^p[I_m,I_m]$.

We now have the following three points. 
\begin{enumerate}[(a)]
\item Take $\alpha \in \Phi^+, \lambda\in X_*(\bT), x\in 1 + \fp^m,$ and $y\in \fp^{m - 1}$.  We have
$$[\lambda(x), u_{\alpha}(y)] = u_{\alpha}\big((x^{\langle\alpha,\lambda\rangle} - 1)y\big),$$
and the valuation of $(x^{\langle\alpha,\lambda\rangle} - 1)y$ is greater than or equal to $2m - 1 > m + e - 1$.  Hence $[\lambda(x), u_{\alpha}(y)]\in I_{m + e}$.  Similarly, we have $[\lambda(x), u_{\alpha}(y)]\in I_{m + e}$ if $\alpha\in \Phi^-, y\in \fp^m$.

\item  Let $\alpha,\beta\in \Phi^+$ and $x,y\in \fp^{m - 1}$, and suppose $\alpha\neq \beta$.  By equation \eqref{comm}, we have
$$[u_\alpha(x),u_{\beta}(y)] = \prod_{\sub{i,j > 0}{i\alpha + j\beta\in \Phi}}u_{i\alpha + j\beta}(c_{\alpha,\beta;i,j}x^iy^j).$$
Every root $i\alpha + j\beta$ appearing above lies in $\Phi^+$, and every element $c_{\alpha,\beta;i,j}x^iy^j$ has valuation which is greater than or equal to $2m - 2 \geq m + e - 1$, so $[u_\alpha(x),u_{\beta}(y)]\in I_{m + e}$.  Similar considerations show that $[u_\alpha(x),u_{\beta}(y)]\in I_{m + e}$ for $\alpha,\beta\in \Phi, \alpha\neq \pm\beta$ and $x,y$ having the appropriate valuations.

\item Now let $\alpha\in \Phi^+$ and $x\in \fp^m, y\in \fp^{m - 1}$.  As before, we have
$$[u_{-\alpha}(x), u_{\alpha}(y)] = u_{-\alpha}(-x^2y(1 - xy)^{-1})\alpha^\vee(1 - xy)u_{\alpha}(xy^2(1 - xy)^{-1}).$$
Examining the valuations of the right-hand side shows $[u_{-\alpha}(x), u_{\alpha}(y)]\in I_{m + e}$.  
\end{enumerate}

Points (a)--(c) imply that $I_m/I_{m + e}$ is abelian, and the equations \eqref{frattini} imply it is $p$-torsion.  Therefore, we obtain $I_m^p[I_m,I_m]\subset I_{m + e}$.  
\end{proof}

\begin{cor}\label{frattinicor}
Suppose that $m > e$ and $I_m$ is torsion-free.  Then $I_m$ is a uniform pro-$p$ group.  
\end{cor}

\begin{proof}
The dimension of $I_m$ as a $p$-adic manifold is $[F:\qp]\dim(\bG)$.  On the other hand, the size of a minimal generating set for $I_m$ is equal to the $\bbF_p$-dimension of the ($\bbF_p$-dual of the) Frattini quotient, which by Lemma \ref{frattinilemma} is equal to
$$\dim_{\bbF_p}\left((1 + \fp^m)/(1 + \fp^{m + e})\right)\dim(\bT) + \dim_{\bbF_p}\left(\fp^{m}/\fp^{m + e}\right)|\Phi| = ef\dim(\bG) = [F:\qp]\dim(\bG).$$
By \cite[Proposition 1.10 and Remark 1.11]{klopschsnopce} this implies $I_m$ is uniform.  
\end{proof}

\bigskip

\section{First cohomology}\label{first}

\textbf{We assume henceforth that $\Phi$ is irreducible.}

We now begin to compute the action of $\cH$ on higher pro-$p$-Iwahori cohomology.  In this section, we compute the Hecke action on $\textnormal{H}^1(I_1,C)$.  By forgetting the $\cH$-module structure, we have an isomorphism of vector spaces $\textnormal{H}^1(I_1,C) = \textnormal{Z}^1(I_1,C) \cong \Hom^{\textnormal{cts}}(I_1,C)$.

\begin{lemma}\label{Tsalphatriv}
Let $\eta\in \textnormal{H}^1(I_1,C) = \Hom^{\textnormal{cts}}(I_1,C)$, and let $\fat{\alpha} = (\alpha,\ell)\in \Pi_\aff$.  If
$$\eta\cdot\T_{\salpha} \neq 0,$$
then $F = \qp$ and $\Phi$ is of type $A_1$.  In this case, $\eta\cdot\T_{\salpha}$ is supported on $u_{\fat{\alpha}}(\fo)$, and 
$$(\eta\cdot\T_{\salpha})(u_{\fat{\alpha}}(y)) = \eta(u_{\fat{\alpha^*}}(-y)),$$
where $\fat{\alpha^*} := (-\alpha,1 - \ell)\in \Pi_\aff$.  
\end{lemma}

\begin{proof}
To determine $\eta\cdot\T_{\salpha}$, it suffices to know its values on $T_1$ and the root subgroups $u_{\fat{\beta}}(\fo)$ for $\fat{\beta}\in\Pi_\aff$.  By equation \eqref{actoncocyc}, we have
$$(\eta\cdot\T_{\salpha})(h) = \sum_{x\in k_F}\eta(\xi_x(h)).$$

If $t\in T_1$, we have
\begin{eqnarray*}
(\eta\cdot\T_{\salpha})(t) & \stackrel{\textnormal{Lem.}~ \ref{conj}\eqref{conjd}}{=} & \sum_{x\in k_F}\eta\left(t^{s_\alpha}u_{-\fat{\alpha}}((1 - \alpha(t)^{-1})[x])\right)\\
 & = & \sum_{x\in k_F}\eta\left(t^{s_\alpha}\right) + \eta\left(u_{-\fat{\alpha}}((1 - \alpha(t)^{-1})[x])\right)\\
 & = & \eta\left(u_{-\fat{\alpha}}\left((1 - \alpha(t)^{-1})\sum_{x\in k_F}[x]\right)\right)\\
 & = & 0.
\end{eqnarray*}

Now let $\fat{\beta} = (\beta,m)\in \Pi_\aff$ and $y\in \fo$.  
\begin{enumerate}[(a)]
\item If $\fat{\beta} = \fat{\alpha}$, then we obtain
\begin{eqnarray*}
(\eta\cdot\T_{\salpha})(u_{\fat{\beta}}(y)) & \stackrel{\textnormal{Lem.}~\ref{conj}\eqref{conja}}{=} & \sum_{x\in k_F}\eta\left(u_{-\fat{\alpha}}([x + \overline{y}] - [x] - y)\right)\\
& = & \eta\left(u_{-\fat{\alpha}}\left(\sum_{x\in k_F}[x + \overline{y}] - [x] - y\right)\right)\\
& = & \eta\left(u_{-\fat{\alpha}}(-qy)\right)\\
& = & \delta_{F,\qp}\eta\left(u_{\fat{\alpha^*}}(-y)\right),
\end{eqnarray*}
where $\fat{\alpha^*} := (-\alpha,1 - \ell)$ is an affine root (which is an element of $\Pi_\aff$ if and only if $\Phi$ is of type $A_1$).  For the last equality, note that $\textnormal{val}(q) = [F:\qp]$.

\item If $\fat{\beta}\neq \fat{\alpha}$ and $\beta\neq -\alpha$, we obtain
\begin{eqnarray*}
(\eta\cdot\T_{\salpha})(u_{\fat{\beta}}(y)) & \stackrel{\textnormal{Lem.}~\ref{conj}\eqref{conjb}}{=} & \sum_{x\in k_F}\eta\bigg(\prod_{\sub{i,j > 0}{i\alpha + j\beta\in \Phi}}u_{s_\alpha(i\alpha + j\beta)}\left(d_{\alpha,i\alpha + j\beta}c_{\alpha,\beta;i,j}\varpi^{-i\ell + jm - j\ell\langle \beta,\alpha^\vee\rangle}[x]^iy^j\right)\\
 & & \qquad \cdot ~u_{s_\alpha(\beta)}\left(d_{\alpha,\beta}\varpi^{m - \ell\langle\beta,\alpha^\vee\rangle}y\right)\bigg)\\
 & = & \sum_{x\in k_F}\sum_{\sub{i,j > 0}{i\alpha + j\beta\in \Phi}}\eta\bigg(u_{s_\alpha(i\alpha + j\beta)}\left(d_{\alpha,i\alpha + j\beta}c_{\alpha,\beta;i,j}\varpi^{-i\ell + jm - j\ell\langle \beta,\alpha^\vee\rangle}[x]^iy^j\right)\bigg)\\
 & = & \sum_{\sub{i,j > 0}{i\alpha + j\beta\in \Phi}}\eta\bigg(u_{s_\alpha(i\alpha + j\beta)}\Big(d_{\alpha,i\alpha + j\beta}c_{\alpha,\beta;i,j}\varpi^{-i\ell + jm - j\ell\langle \beta,\alpha^\vee\rangle}y^j\sum_{x\in k_F}[x]^i\Big)\bigg)\\
  & = & 0.
\end{eqnarray*}
The last equality follows from the fact that $x\longmapsto [x]^i$ is a nontrivial character (note that $1\leq i\leq 3$ and $p\geq 5$).

\item Finally, assume that $\fat{\beta}\neq\fat{\alpha}$ and $\beta = -\alpha$, so that $\fat{\beta} = \fat{\alpha^*} = (-\alpha,1 - \ell)$.  For $x\in k_F$, set $\nu_x := 1 + \varpi[x]y$.  We then have
\begin{eqnarray*}
(\eta\cdot\T_{\salpha})(u_{\fat{\beta}}(y)) & \stackrel{\textnormal{Lem.}~\ref{conj}\eqref{conjc}}{=} & \sum_{x\in k_F}\eta\left(u_{\fat{\alpha}}(-\varpi y\nu_x^{-1})\alpha^\vee(\nu_x^{-1})u_{\fat{\beta}}([x]^2y\nu_x^{-1})\right)\\
& = & \sum_{x\in k_F}\eta\left(u_{\fat{\beta}}([x]^2y)\right)\\
& = & \eta\left(u_{\fat{\beta}}\left(y\sum_{x\in k_F}[x]^2\right)\right)\\
& = & 0,
\end{eqnarray*}
where again the last equality follows from the fact that $x\longmapsto [x]^2$ is a nontrivial character for $p\geq 5$.
\end{enumerate}
\end{proof}

\begin{lemma}\label{TomegaonT}
Let $\eta\in \Hom^{\textnormal{cts}}(\overline{T}_1,C)\subset \textnormal{H}^1(I_1,C)$ and $\omega\in \widetilde{\Omega}$.  Then
$$\eta\cdot \T_{\omega} = \eta.$$
\end{lemma}

\begin{proof}
Since $\omega$ normalizes $I_1$, the simplified version of equation \eqref{actoncocyc} gives
$$(\eta\cdot\T_{\omega})(h) = \eta(\omega h\omega^{-1})$$
for $h\in I_1$.  If $h\in U_\alpha$ for some $\alpha\in \Phi$, then $\omega h\omega^{-1}\in U_{\bo(\alpha)}$, which implies $\eta\cdot\T_\omega$ is supported in $T_1$.  Taking $h = t\in T_1$ and writing $\bo$ as a product of simple reflections, we see that 
$$\omega t\omega^{-1} = \bo t\bo^{-1} = tt',$$
where $t'\in \langle\alpha^\vee(1 + \fp)\rangle_{\alpha\in\Phi}$.  Therefore 
$$(\eta\cdot\T_{\omega})(t) = \eta(tt') = \eta(t),$$
which gives the claim.  
\end{proof}

\begin{lemma}\label{Tomegaonunip}
Let $\fat{\alpha} = (\alpha,\ell)\in \Pi_\aff$ and $0\leq r \leq f - 1$.  
\begin{enumerate}[(a)]
\item If $t_0\in T_0$, then
$$\eta_{\fat{\alpha},r}\cdot \T_{t_0} = \overline{\alpha}(t_0)^{p^r}\cdot\eta_{\fat{\alpha},r}.$$
\label{Tomegaonunipa}
\item If $\omega\in \widetilde{\Omega}$ and $\omega\widehat{\bo}^{-1} = \lambda(\varpi)$ for some $\lambda\in X_*(\bT)$, then
$$\eta_{\fat{\alpha},r}\cdot \T_{\omega} = d_{\bo,\bo^{-1}(\alpha)}\cdot\eta_{\omega^{-1}(\fat{\alpha}),r}.$$
\label{Tomegaonunipb}
\end{enumerate}
\end{lemma}

\begin{proof}
Given an arbitrary $\omega\in \widetilde{\Omega}$, the simplified version of equation \eqref{actoncocyc} gives
$$(\eta_{\fat{\alpha},r}\cdot\T_{\omega})(h) = \eta_{\fat{\alpha},r}(\omega h\omega^{-1})$$
for $h\in I_1$.  If $h\in T_1$, then $\omega h \omega^{-1}\in T_1$, which implies $\eta_{\fat{\alpha},r}\cdot\T_{\omega}$ takes the value 0 on $T_1$.  The image of $\omega$ in $W$ acts on $\Pi_\aff$, and therefore $\eta_{\fat{\alpha},r}\cdot\T_{\omega}$ is supported on $u_{\omega^{-1}(\fat{\alpha})}(\fo)$.  
\begin{enumerate}[(a)]
\item Taking $\omega = t_0\in T_0$ and $y\in \fo$, we have
\begin{eqnarray*}
(\eta_{\fat{\alpha},r}\cdot\T_{t_0})(u_{\fat{\alpha}}(y)) & = & \eta_{\fat{\alpha},r}\left(t_0u_{\fat{\alpha}}(y)t_0^{-1}\right)\\
 & = & \eta_{\fat{\alpha},r}\left(u_{\fat{\alpha}}(\alpha(t_0)y)\right)\\
 & = & \overline{\alpha}(t_0)^{p^r}\eta_{\fat{\alpha},r}\left(u_{\fat{\alpha}}(y)\right).
\end{eqnarray*}
\item If $\omega\widehat{\bo}^{-1} = \lambda(\varpi)$ for some $\lambda\in X_*(\bT)$ and $y\in \fo$, then
\begin{eqnarray*}
(\eta_{\fat{\alpha},r}\cdot\T_{\omega})\left(u_{\omega^{-1}(\fat{\alpha})}(y)\right) & = & \eta_{\fat{\alpha},r}\left(\lambda(\varpi)\widehat{\bo}u_{\omega^{-1}(\fat{\alpha})}(y)\widehat{\bo}^{-1}\lambda(\varpi)^{-1}\right)\\
 & = & \eta_{\fat{\alpha},r}\left(u_{\fat{\alpha}}(d_{\bo,\bo^{-1}(\alpha)}y)\right)\\
 & = & d_{\bo,\bo^{-1}(\alpha)}\cdot\eta_{\omega^{-1}(\fat{\alpha}),r}\left(u_{\omega^{-1}(\fat{\alpha})}(y)\right)
\end{eqnarray*}
(for the last step, recall that $d_{\bo,\bo^{-1}(\alpha)} = \pm 1$).
\end{enumerate}
\end{proof}

\begin{thm}
\label{mainthmtrivchar}
Set $a := \textnormal{dim}_{C}(\Hom^{\textnormal{cts}}(\overline{T}_1,C))$.
\begin{enumerate}[(a)]
\item  If $F = \qp$ and $\Phi$ is of type $A_1$, then we have an isomorphism of $\cH$-modules
$$\textnormal{H}^1(I_1,C) \cong \chi_{\textnormal{triv}}^{\oplus a}\oplus \ind_{\cH_T}^{\cH}(\overline{\alpha}),$$
where $\Pi = \{\alpha\}$, and where $\ind_{\cH_T}^{\cH}(\overline{\alpha})$ is the simple $\cH$-module given by $\ind_B^G(\overline{\alpha})^{I_1}$ (cf. \cite[Proposition 4.4]{olliviervigneras}).  
\item If $F \neq \qp$ or $\Phi$ is not of type $A_1$, then we have an isomorphism of $\cH$-modules
$$\textnormal{H}^1(I_1,C) \cong \chi_{\textnormal{triv}}^{\oplus a}\oplus \bigoplus_{\fat{\alpha}\in \Pi_{\aff}/\Omega}\bigoplus_{r = 0}^{f - 1}\fm_{\fat{\alpha},r},$$
where $\fm_{\fat{\alpha},r}$ is a simple supersingular $\cH$-module, of dimension equal to the size of the orbit of $\fat{\alpha}$ under $\Omega$.  Furthermore, all elements $\T_{\widehat{s_{\fat{\beta}}}}$ for $\fat{\beta}\in \Pi_\aff$ act by $0$ on $\fm_{\fat{\alpha},r}$, while the elements $\T_{t_0}$ for $t_0\in T_0$ act semisimply with eigenvalues $\overline{\alpha}(\bo t_0 \bo^{-1})^{p^r}$, as $\omega$ ranges over $\widetilde{\Omega}$ modulo the stabilizer of $\fat{\alpha}$.  
\end{enumerate}
\end{thm}

\begin{proof}
This follows from Corollary \ref{h1basis} and the results of this section.  More precisely, Lemmas \ref{Tsalphatriv} and \ref{TomegaonT} show that the homomorphisms coming from $\overline{T}_1$ all give the trivial $\cH$-character.  When $F = \qp$ and $\Phi$ is of type $A_1$, Lemmas \ref{Tsalphatriv} and \ref{Tomegaonunip} give the action of $\cH$ on $\textnormal{span}\{\eta_{(\alpha,0),0},~ \eta_{(-\alpha,1),0}\}$ and the proof of Proposition 5.5 of \cite{koziol:h1ps} shows that 
$$\textnormal{span}\{\eta_{(\alpha,0),0},~ \eta_{(-\alpha,1),0}\} \cong \ind_{\cH_T}^{\cH}(\overline{\alpha})$$ 
as $\cH$-modules.  On the other hand, if $F\neq \qp$ or $\Phi$ is not of type $A_1$, then by Lemmas \ref{Tsalphatriv} and \ref{Tomegaonunip}\eqref{Tomegaonunipa} the $C$-vector space spanned by $\eta_{\fat{\alpha},r}$, for $\fat{\alpha} = (\alpha,\ell)\in \Pi_\aff$ and $0\leq r \leq f - 1$, is stable by $\cH_\aff$.  Moreover, it cannot be equal to a twist of the trivial or sign character (since all $\T_{\widehat{s_{\fat{\beta}}}}$ act by $0$ on $\eta_{\fat{\alpha},r}$, and $\eta_{\fat{\alpha},r}\cdot\T_{\alpha^\vee([x])} = x^{2p^r}\cdot\eta_{\fat{\alpha},r}\neq \eta_{\fat{\alpha},r}$ for $x\in k_F^\times$).  Therefore the $\cH$-module generated by $\eta_{\fat{\alpha},r}$ is supersingular, and spanned by $\{\eta_{\fat{\beta},r}\}$ where $\fat{\beta}$ is in the orbit of $\fat{\alpha}$ under $\Omega$.  If the characters $\overline{\beta}$ for $\beta\in \Pi\sqcup \{-\alpha_0\}$ are all distinct, then the $\cH$-module generated by $\eta_{\fat{\alpha},r}$ is simple.  This happens if $\Phi$ is not of type $A_1$, or if $q \geq 7$.  The remaining case ($\Phi$ of type $A_1$ and $q = 5$) is easily dealt with by direct calculation.  
\end{proof}

\begin{thm}\label{arbrootsys}
Suppose $p\nmid |\pi_1(\bG/\bZ)|$, assume the root system $\Phi$ is arbitrary, and write $\Phi = \Phi^{(1)}\sqcup \ldots \sqcup \Phi^{(d)}$ in terms of irreducible components.  Assume further that either (1) $F\neq \qp$; or (2) $\Phi$ does not contain an $A_1$ component.  Then $\textnormal{H}^d(I_1,C)$ contains a supersingular $\cH$-module, and $\textnormal{H}^n(I_1,C)$ contains no supersingular subquotients for $0\leq n < d$.  
\end{thm}

\begin{proof}
We proceed in three steps.

\begin{enumerate}[(a)]
\item  Suppose first that $\bG$ is semisimple and simply connected.  Then $\bG$ factors as a direct product $\bG = \bG^{(1)}\times \cdots \times \bG^{(d)}$, where $\bG^{(i)}$ corresponds to the root system $\Phi^{(i)}$, and consequently we have a factorization 
$$I_1 = I_1^{(1)}\times \cdots \times I_1^{(d)}.$$
Moreover, the algebra $\cH$ is equal to $\cH_\aff$, and decomposes as a direct product $\cH = \cH^{(1)}\otimes_C \cdots \otimes_C\cH^{(d)}$, where we may identify $\cH^{(i)}$ as the subalgebra of functions with support in $G^{(i)}I_1$.   In particular, the subalgebra generated by $\cH^{(i)}$ and $\{\T_{t_0}\}_{t_0\in T_0}$ is an irreducible component of $\cH$ (as defined in Subsection \ref{ssectdefs}).

The tracing through the arguments in \cite[Section 2.4, Exercise 7]{nsw:coh} shows that the isomorphism in the K\"unneth formula is given by the cross product:
\begin{eqnarray}
\bigoplus_{i_1 + \ldots + i_d = n} \textnormal{H}^{i_1}(I_1^{(1)},C)\otimes_C\cdots \otimes_C \textnormal{H}^{i_d}(I_1^{(d)},C) & \stackrel{\sim}{\longrightarrow} & \textnormal{H}^{n}(I_1,C) \label{kunneth}\\
{\left[f^{(1)}\right]}\otimes \cdots \otimes \left[f^{(d)}\right] & \longmapsto & \left[f^{(1)} \times  \cdots \times f^{(d)}\right] \notag
\end{eqnarray}
where, for $[f^{(1)}]\otimes \cdots \otimes [f^{(d)}] \in \textnormal{H}^{i_1}(I_1^{(1)},C)\otimes_C\cdots \otimes_C \textnormal{H}^{i_d}(I_1^{(d)},C)$, the cocycle representing the cross product is defined by
\begin{flushleft}
$\left(f^{(1)} \times f^{(2)}\times \cdots \times f^{(d)}\right)\left((g_1^{(1)},\ldots, g_1^{(d)}),~(g_2^{(1)},\ldots, g_2^{(d)}),~\ldots,~ (g_n^{(1)},\ldots, g_n^{(d)})\right)$ 
\end{flushleft}
\begin{flushright}
$= f^{(1)}\left(g_1^{(1)},\ldots, g_{i_1}^{(1)}\right)\cdot f^{(2)}\left(g_{i_1 + 1}^{(2)},\ldots, g_{i_1 + i_2}^{(2)}\right)\cdots f^{(d)}\left(g_{i_1 + \cdots + i_{d - 1} + 1}^{(d)},\ldots, g_{i_1 + \cdots + i_{d}}^{(d)}\right)$,
\end{flushright}
where $g_{j}^{(i)}\in I_1^{(i)}$.  Hence, using the decomposition $\cH = \cH^{(1)}\otimes_C\cdots \otimes_C\cH^{(d)}$, we see that the isomorphism in equation \eqref{kunneth} is actually $\cH$-equivariant.

Taking $i_1 = i_2 = \ldots = i_d = 1$ gives an injection
$$\textnormal{H}^{1}(I_1^{(1)},C)\otimes_C\cdots \otimes_C \textnormal{H}^{1}(I_1^{(d)},C)  \longhookrightarrow  \textnormal{H}^{d}(I_1,C).$$
By Theorem \ref{mainthmtrivchar} and the assumptions of the present theorem, the left-hand side contains a tensor product of supersingular characters of each $\cH^{(i)}$.  Therefore, the restriction to each irreducible component of $\cH$ of this tensor product is a supersingular module, and thus $\textnormal{H}^d(I_1,C)$ contains a supersingular $\cH$-module.

If $0\leq n < d$, then any $d$-tuple $(i_1,\ldots, i_d)$ which satisfies $i_1 + \ldots + i_d = n$ must have some $i_j = 0$.  The restriction to $\cH^{(j)}$ of the module
$$\textnormal{H}^{i_1}(I_1^{(1)},C)\otimes_C\cdots \otimes_C\textnormal{H}^0(I_1^{(j)},C)\otimes_C\cdots \otimes_C \textnormal{H}^{i_d}(I_1^{(d)},C)\cong \textnormal{H}^{i_1}(I_1^{(1)},C)\otimes_C\cdots \otimes_C \chi_{\textnormal{triv}}^{(j)} \otimes_C\cdots \otimes_C \textnormal{H}^{i_d}(I_1^{(d)},C)$$
is a direct sum of trivial characters of $\cH^{(j)}$, and therefore the tensor product cannot contain a supersingular $\cH$-module.  Using equation \eqref{kunneth}, we get that $\textnormal{H}^n(I_1,C)$ contains no supersingular subquotients if $0 \leq n < d$.

\item  Suppose next that $\bG$ has simply connected derived subgroup $\bG^{\textnormal{der}}$.  Then $\bT^{\textnormal{der}} := \bG^{\textnormal{der}}\cap \bT$ is a maximal torus of $\bG^{\textnormal{der}}$ and $I_1^{\textnormal{der}} := G^{\textnormal{der}}\cap I_1$ is a pro-$p$-Iwahori subgroup of $G^{\textnormal{der}}$.  We let $T_1^{\textnormal{der}}$ (resp. $Z_1$) denote the maximal pro-$p$ subgroup of $T^{\textnormal{der}}$ (resp. of the connected center $Z$).

Let $n'$ denote the order of $\pi_1(\bG/\bZ) = X_*(\bT)/(X_*(\bT^{\textnormal{der}})\oplus X_*(\bZ))$.  Then given any $\lambda\in X_*(\bT)$, we may write $n'\lambda = \lambda^{\textnormal{der}} + \lambda^Z$ with $\lambda^{\textnormal{der}}\in X_*(\bT^{\textnormal{der}})$ and $\lambda^Z\in X_*(\bZ)$.  Since $p\nmid n'$, the map $x\longmapsto x^{n'}$ gives an automorphism of $1 + \fp$, which implies
$$\lambda(x) = \lambda(x^{1/n'})^{n'} = (n'\lambda)(x^{1/n'}) = (\lambda^{\textnormal{der}} + \lambda^Z)(x^{1/n'}) = \lambda^{\textnormal{der}}(x^{1/n'})\lambda^Z(x^{1/n'})$$
for $x\in 1 + \fp$.  Since $T_1$ is generated by elements of the form $\lambda(x)$ for $\lambda\in X_*(\bT)$ and $x\in 1 + \fp$, the above equation shows
$$T_1 = T_1^{\textnormal{der}}Z_1,$$
and the Iwahori decomposition implies that the multiplication map
$$Z_1\times I_1^{\textnormal{der}}\longrightarrow I_1$$
is surjective.  Its kernel is equal to the finite $p$-group $Z_1\cap I_1^{\textnormal{der}}$, which is a subgroup of the finite group $\bZ\cap \bG^{\textnormal{der}}$ (we identify algebraic groups with their groups of rational points over an algebraic closure of $F$).  However, $\bG^{\textnormal{der}}$ is the simply connected cover of the semisimple group $\bG/\bZ$, and therefore the kernel $\bZ\cap \bG^{\textnormal{der}}$ of the surjection $\bG^{\textnormal{der}} \longtwoheadrightarrow \bG/\bZ$ is equal to $\pi_1(\bG/\bZ)$.  Since $Z_1\cap I_1^{\textnormal{der}}$ is a $p$-group contained in $\pi_1(\bG/\bZ)$, and the order of the latter group is prime to $p$, we conclude that $Z_1\cap I_1^{\textnormal{der}} = 1$, and therefore
$$Z_1\times I_1^{\textnormal{der}} \cong I_1.$$

Now let $\cH^\textnormal{der}$ denote the pro-$p$-Iwahori--Hecke algebra of $G^{\textnormal{der}}$ with respect to $I_1^{\textnormal{der}}$.  We have an injection of algebras $\cH^{\textnormal{der}}\longhookrightarrow \cH$ (\cite[Proposition 4.3]{koziol:glnsln}), sending the characteristic function of $I_1^{\textnormal{der}}gI_1^{\textnormal{der}}$ to the characteristic function of $I_1gI_1$ for $g\in G^{\textnormal{der}}$.  Using this injection, we have a chain of subalgebras $\cH^{\textnormal{der}}\subset \cH_\aff \subset \cH$.  Furthermore, using the decomposition $I_1\cong Z_1\times I_1^{\textnormal{der}}$ and the K\"unneth formula above, we have an isomorphism
$$\bigoplus_{i + j = n}\textnormal{H}^i(Z_1,C)\otimes_C\textnormal{H}^j(I_1^{\textnormal{der}},C)\stackrel{\sim}{\longrightarrow} \textnormal{H}^n(I_1,C)$$
which is $\cH^{\textnormal{der}}$-equivariant (where $\cH^{\textnormal{der}}$ acts on the left-hand side on $\textnormal{H}^j(I_1^{\textnormal{der}},C)$, and on the right-hand side via the injection $\cH^{\textnormal{der}}\longhookrightarrow \cH$).  It is straightforward to check that a simple $\cH$-module is supersingular if and only if its restriction to $\cH^{\textnormal{der}}$ contains a supersingular $\cH^{\textnormal{der}}$-character (cf. Subsection \ref{ssingsect}).  
Therefore, by part (a) we see that $\textnormal{H}^d(I_1,C)$ contains a supersingular $\cH$-module, while $\textnormal{H}^n(I_1,C)$ contains no supersingular subquotients if $0 \leq n < d$.

\item Finally, let $\bG$ be arbitrary, and let $\widetilde{\bG}$ denote a $z$-extension of $\bG$.  This means that $\widetilde{\bG}$ is a split connected reductive group with simply connected derived subgroup, which fits into a short exact sequence
$$1 \longrightarrow \widetilde{\bZ} \longrightarrow \widetilde{\bG} \longrightarrow \bG \longrightarrow 1$$
where $\widetilde{\bZ}$ is a central torus.  This gives the short exact sequence
$$1 \longrightarrow \tZ_1 \longrightarrow \tI_1 \longrightarrow I_1 \longrightarrow 1,$$
where $\tZ_1$ is the maximal pro-$p$ subgroup of $\tZ$, and $\tI_1$ is the pro-$p$-Iwahori subgroup defined by the same chamber as $I_1$ (note that we may identify the semisimple Bruhat--Tits buildings of $\widetilde{\bG}$ and $\bG$; see \cite[Section 3.5(ii)]{henniartvigneras:satake}).  We let $\widetilde{\cH}$ denote the pro-$p$-Iwahori--Hecke algebra of $\tG$ with respect to $\tI_1$, and let $\widetilde{\T}_g$ denote the characteristic function of $\tI_1g\tI_1$.  Taking the quotient of $\widetilde{\cH}$ by the two-sided ideal generated by $(\widetilde{\T}_{z} - 1)_{z\in \tZ}$ gives an isomorphism
$$\widetilde{\cH}/(\widetilde{\T}_{z} - 1)_{z\in \tZ} \cong \cH.$$

The short exact sequence above gives rise to a Hochschild--Serre spectral sequence
$$E_2^{i,j} = \textnormal{H}^i(I_1,\textnormal{H}^j(\tZ_1,C))\Longrightarrow \textnormal{H}^{i + j}(\tI_1,C),$$
which is $\cH$-equivariant (the algebra $\widetilde{\cH}$ acts on the right-hand side and the elements $\widetilde{\T}_z, z\in \tZ$ act trivially, so the action descends to $\cH$).  Since $\tZ_1$ is central, the group $I_1$ acts trivially on $\textnormal{H}^j(Z_1,C)$, and therefore the spectral sequence above becomes
$$E_2^{i,j} = \textnormal{H}^i(I_1,C)\otimes_C\textnormal{H}^j(\tZ_1,C)\Longrightarrow \textnormal{H}^{i + j}(\tI_1,C).$$
Since $\textnormal{H}^0(I_1,C)$ is the trivial character of $\cH$, it is nonsupersingular.  Suppose that $\textnormal{H}^i(I_1,C)$ has no supersingular subquotients for $0\leq i \leq n - 1 < d - 1$.  Then the same is true of the $\cH$-module $E_2^{i,j}$ for all $j$.  The $\cH$-module $E_\infty^{n,0}$ is the quotient of $E_2^{n,0}\cong \textnormal{H}^{n}(I_1,C)$ by various nonsupersingular subquotients of $E_2^{i,j}$ with $0\leq i \leq n - 1$.  On the other hand, $E_\infty^{n,0}$ is a submodule of $\textnormal{H}^{n}(\tI_1,C)$, which has no supersingular subquotients by part (b).  We conclude that $\textnormal{H}^{n}(I_1,C)$ has no supersingular subquotients for all $0\leq n < d$.  Using this fact, the filtration on $\textnormal{H}^{d}(\tI_1,C)$ coming from the spectral sequence, and that $\textnormal{H}^{d}(\tI_1,C)$ contains a supersingular module, one obtains that $\textnormal{H}^{d}(I_1,C)$ contains a supersingular $\cH$-module.  
\end{enumerate}
\end{proof}

\begin{rmk}
If the assumptions of the above theorem are not met, then it may happen that $\textnormal{H}^n(I_1,C)$ has no supersingular subquotients for all $n$.  Consider the following example.  Let $\bG = \bS\bL_2$ over $\qp$, and $I_1$ the ``upper triangular mod-$p$'' pro-$p$-Iwahori subgroup of $G = \textnormal{SL}_2(\qp)$.  Since $p\geq 5$, the group $I_1$ is torsion-free, and it is therefore a Poincar\'e group of dimension 3 (\cite[Section 4.5, Remark 3]{serre:galoiscoh}).  Write $\Pi = \{\alpha\}$; we then have
\begin{eqnarray*}
\textnormal{H}^0(I_1,C) & \cong & \chi_{\textnormal{triv}},\\
\textnormal{H}^1(I_1,C) & \cong & \ind_{\cH_T}^{\cH}(\overline{\alpha}),\\
\textnormal{H}^2(I_1,C) & \cong & \ind_{\cH_T}^{\cH}(\overline{\alpha}),\\
\textnormal{H}^3(I_1,C) & \cong & \chi_{\textnormal{triv}},\\
\textnormal{H}^n(I_1,C) & \cong & 0~\textnormal{for}~n \geq 4.
\end{eqnarray*}
The first isomorphism follows from the remark in Subsection \ref{heckecoh}, the second from Theorem \ref{mainthmtrivchar}, and the last two from Theorem \ref{htop} below and the definition of Poincar\'e groups (and the universal coefficient theorem).

We sketch a proof of the isomorphism $\textnormal{H}^2(I_1,C)  \cong  \ind_{\cH_T}^{\cH}(\overline{\alpha})$.  By the universal coefficient theorem, it suffices to prove the claim with $C$ replaced by $\bbF_p$.  Let $\alpha^*$ denote the affine root $(-\alpha,1)$.  Then $\widehat{s_\alpha} = \sm{0}{1}{-1}{0}$ and $\widehat{s_{\alpha^*}} = \sm{0}{-p^{-1}}{p}{0}$, and we set $K_1 := I_1\cap \widehat{s_\alpha}I_1 \widehat{s_\alpha}^{-1}$ and $K_1^* := I_1 \cap \widehat{s_{\alpha^*}}I_1 \widehat{s_{\alpha^*}}^{-1}$.  We then have
$$K_1 = \begin{pmatrix}1 + p\bbZ_p & p\bbZ_p \\ p\bbZ_p & 1 + p\bbZ_p\end{pmatrix}\cap \textnormal{SL}_2(\qp)~\textnormal{and}~K_1^* = \begin{pmatrix}1 + p\bbZ_p & \bbZ_p \\ p^2\bbZ_p & 1 + p\bbZ_p\end{pmatrix}\cap \textnormal{SL}_2(\qp),$$
and both $K_1$ and $K_1^*$ are uniform pro-$p$ groups.  Consequently, \cite[Theorem 5.1.5]{symondsweigel} implies that the cohomology algebra of $K_1$ (resp. $K_1^*$) is an exterior algebra on $\textnormal{H}^1(K_1,\bbF_p)$ (resp. $\textnormal{H}^1(K_1^*,\bbF_p)$).  We have $\textnormal{H}^1(K_1,\bbF_p) = \Hom^{\textnormal{cts}}(K_1,\bbF_p) = \bbF_p\eta_{\textnormal{u}} \oplus \bbF_p\eta_{\textnormal{d}} \oplus \bbF_p\eta_{\textnormal{l}}$, where for $g = \sm{1 + a}{b}{c}{1 + d}, a,b,c,d\in p\bbZ_p$, we define
$$\eta_{\textnormal{u}}(g) = \overline{p^{-1}b},\qquad \eta_{\textnormal{d}}(g) = \overline{p^{-1}a},\qquad \eta_{\textnormal{l}}(g) = \overline{p^{-1}c}$$
(we have a similar description of $\textnormal{H}^1(K_1^*,\bbF_p)$, with basis $\{\eta_{\textnormal{u}}^*,~ \eta_{\textnormal{d}}^*,~ \eta_{\textnormal{l}}^*\}$).  Now set $\kappa_1 := \textnormal{cor}_{K_1}^{I_1}(\eta_{\textnormal{u}}\smile \eta_{\textnormal{d}}), \kappa_2 :=  \textnormal{cor}_{K_1^*}^{I_1}(\eta_{\textnormal{l}}^*\smile \eta_{\textnormal{d}}^*)\in \textnormal{H}^2(I_1,\bbF_p)$.  Note that these elements are nonzero: by the projection formula, we have $\kappa_1\smile \eta_{(-\alpha,1),0} = \textnormal{cor}_{K_1}^{I_1}(\eta_{\textnormal{u}}\smile \eta_{\textnormal{d}}\smile\eta_{\textnormal{l}}) \neq 0$, since $\eta_{\textnormal{u}}\smile \eta_{\textnormal{d}}\smile\eta_{\textnormal{l}}$ is a basis for $\textnormal{H}^3(K_1,\bbF_p)$ and the corestriction gives an isomorphism $\textnormal{cor}_{K_1}^{I_1}:\textnormal{H}^3(K_1,\bbF_p) \stackrel{\sim}{\longrightarrow}\textnormal{H}^3(I_1,\bbF_p)$ (and similarly for $K_1^*$; cf. \cite[Proof of Proposition 30]{serre:galoiscoh}).  Moreover, since $I_1, K_1$ and $K_1^*$ are normalized by $T_0$, the conjugation action commutes with corestriction and cup product, and we see that $T_0$ acts on $\kappa_1$ (resp. $\kappa_2$) by $t_{0,*}^{-1}\kappa_1 = \overline{\alpha}(t_0)\cdot\kappa_1$ (resp. $t_{0,*}^{-1}\kappa_2 = \overline{\alpha}(t_0)^{-1}\cdot\kappa_2$).

It suffices to compute the actions of $\T_{\widehat{s_\alpha}}$ and $\T_{\widehat{s_{\alpha^*}}}$.  We have
\begin{eqnarray*}
\kappa_1\cdot\T_{\widehat{s_\alpha}} & = & \textnormal{cor}_{K_1}^{I_1}\circ \widehat{s_\alpha}^{-1}_*\circ\textnormal{res}_{K_1}^{I_1}\circ\textnormal{cor}_{K_1}^{I_1}(\eta_{\textnormal{u}}\smile\eta_{\textnormal{d}})\\
 & = & \textnormal{cor}_{K_1}^{I_1}\circ \widehat{s_\alpha}^{-1}_*\circ\left(\sum_{x\in \bbF_p}u_\alpha([x])_*\right)(\eta_{\textnormal{u}}\smile\eta_{\textnormal{d}})\\
 & = & 0,
 \end{eqnarray*}
where the last equality follows from the fact that $u_\alpha([x])_*(\eta_{\textnormal{u}}\smile\eta_{\textnormal{d}}) = \eta_{\textnormal{u}}\smile\eta_{\textnormal{d}} - x\cdot\eta_{\textnormal{u}}\smile\eta_{\textnormal{l}} - x^2\cdot \eta_{\textnormal{d}}\smile\eta_{\textnormal{l}}$ and $p\geq 5$.  Computing on the level of cocycles gives $\textnormal{res}_{K_1^*}^{I_1}\circ\textnormal{cor}_{K_1}^{I_1}(\eta_{\textnormal{u}}\smile\eta_{\textnormal{d}}) = \eta_{\textnormal{u}}^*\smile\eta_{\textnormal{d}}^*$, and therefore
\begin{eqnarray*}
\kappa_1\cdot\T_{\widehat{s_{\alpha^*}}} & = & \textnormal{cor}_{K_1^*}^{I_1}\circ \widehat{s_{\alpha^*}}^{-1}_*\circ\textnormal{res}_{K_1^*}^{I_1}\circ\textnormal{cor}_{K_1}^{I_1}(\eta_{\textnormal{u}}\smile\eta_{\textnormal{d}})\\
 & = & \textnormal{cor}_{K_1^*}^{I_1}\circ \widehat{s_{\alpha^*}}^{-1}_*(\eta_{\textnormal{u}}^*\smile\eta_{\textnormal{d}}^*)\\
 & = & \textnormal{cor}_{K_1^*}^{I_1}\left((-\eta_{\textnormal{l}}^*)\smile(-\eta_{\textnormal{d}}^*)\right)\\
 & = & \kappa_2.
\end{eqnarray*}
Likewise, we have $\kappa_2\cdot \T_{\widehat{s_\alpha}} = \kappa_1$ and $\kappa_2\cdot \T_{\widehat{s_{\alpha^*}}} = 0$.  This shows that $\kappa_1$ and $\kappa_2$ are linearly independent, and thus form a basis for $\textnormal{H}^2(I_1,\bbF_p)$ (recall that by the definition of Poincar\'e groups, we have $\dim_{\bbF_p}(\textnormal{H}^2(I_1,\bbF_p)) = \dim_{\bbF_p}(\textnormal{H}^1(I_1,\bbF_p)) = 2$).  Comparing the above calculations with the action of $\cH$ on $\textnormal{H}^1(I_1,C)$ gives the claim.

Let us remark that there are several ways of proving the isomorphism $\textnormal{H}^2(I_1,C)  \cong  \ind_{\cH_T}^{\cH}(\overline{\alpha})$: (1) Using the equivalence of categories between $\textnormal{SL}_2(\qp)$-representations and $\cH$-modules, we could employ same techniques as in \cite[Proposition 10.1]{paskunas:montrealfunctor}; (2)  The Bockstein map $\beta: \textnormal{H}^1(I_1,\bbF_p) \longrightarrow \textnormal{H}^2(I_1,\bbF_p)$ is an injective map between vector spaces of the same dimension, which is furthermore equivariant for the operators $\T_w$; (3) By applying Theorem \ref{poincare}, we have $\textnormal{H}^2(I_1,C) \cong \textnormal{H}^1(I_1,C)^* \cong   \ind_{\cH_T}^{\cH}(\overline{\alpha})^* \cong  \ind_{\cH_T}^{\cH}(\overline{\alpha})$, where the last isomorphism follows from \cite[Theorem 4.9]{abe:inv}.  
\end{rmk}

\begin{rmk}
For a slightly less trivial example of the above phenomenon, let $\bG^{(1)}$ denote any split connected reductive group over $\qp$, and $I_1^{(1)}$ a pro-$p$-Iwahori subgroup of $G^{(1)} = \bG^{(1)}(\qp)$.  If we consider the group $\bS\bL_2\times \bG^{(1)}$, then by the K\"unneth formula we have 
$$\bigoplus_{i + j = n}\textnormal{H}^i(I_1,C)\otimes_C\textnormal{H}^j(I_1^{(1)},C)\cong \textnormal{H}^n(I_1\times I_1^{(1)},C),$$
and the isomorphism is equivariant for the pro-$p$-Iwahori--Hecke algebra of $\textnormal{SL}_2(\qp)\times G^{(1)}$ with respect to $I_1\times I_1^{(1)}$.  Since the first tensor factor on the left-hand side is never supersingular by the remark above, we see that $\textnormal{H}^n(I_1\times I_1^{(1)},C)$ contains no supersingular subquotients for all $n$.  
\end{rmk}

\bigskip

\section{Top cohomology and duality}\label{top}

We continue to assume $\Phi$ is irreducible, and assume in this section that $I_1$ is torsion-free.  By \cite[Section 4.5]{serre:galoiscoh}, this assumption guarantees that the cohomological dimension of $I_1$ is finite, equal to $[F:\qp]\dim(\bG)$, the dimension of $I_1$ as a $p$-adic manifold.

\begin{thm}\label{htop}
Suppose $I_1$ is torsion-free, and let $n := [F:\qp]\dim(\bG)$ denote the cohomological dimension of $I_1$.  We then have an isomorphism of $\cH$-modules
$$\textnormal{H}^n(I_1,C)\cong \chi_{\textnormal{triv}}.$$
\end{thm}

\begin{proof}
By the universal coefficient theorem, it suffices to prove $\textnormal{H}^n(I_1,\bbF_p)$ is isomorphic to the trivial character of $\cH$ over $\bbF_p$.  Since $I_1$ is a Poincar\'e group of dimension $n$, $\textnormal{H}^n(I_1,\bbF_p)$ is a one-dimensional $\bbF_p$-vector space.

Suppose first that $w\in \tW$ with $\boldl(w) > 0$.  Since $[I_1:I_1\cap wI_1 w^{-1}] = q^{\boldl(w)}$, the subgroup $I_1\cap wI_1 w^{-1}$ is a proper subgroup of $I_1$.  Therefore, by \cite[Section 4.5, Exercise 5]{serre:galoiscoh}, the restriction map $\textnormal{res}^{I_1}_{I_1\cap wI_1 w^{-1}}: \textnormal{H}^n(I_1,\bbF_p) \longrightarrow \textnormal{H}^n(I_1\cap wI_1 w^{-1},\bbF_p)$ is the zero map.  By equation \eqref{heckefunctorial}, the element $\T_w$ acts by 0 on $\textnormal{H}^n(I_1,\bbF_p)$.

We next compute the action of $\T_\omega, \omega\in \widetilde{\Omega}$ on $\textnormal{H}^n(I_1,\bbF_p)$, which, by equation \eqref{heckefunctorial}, is given by $\omega^{-1}_*:\textnormal{H}^n(I_1,\bbF_p)\longrightarrow \textnormal{H}^n(I_1,\bbF_p)$.  We use the following reduction.  Let $e$ denote the absolute ramification index of $F$, and let $m > e$.  Since $I_m$ is an open subgroup of $I_1$, the corestriction map $\textnormal{cor}_{I_m}^{I_1}: \textnormal{H}^n(I_m,\bbF_p) \longrightarrow \textnormal{H}^n(I_1,\bbF_p)$ is an isomorphism (\cite[Proof of Proposition 30]{serre:galoiscoh}).  Since $\omega$ normalizes $I_1$, it also normalizes $I_m$ (\cite[p. 114]{ss:reptheorysheaves}), and the map $\omega_*^{-1}$ commutes with $\textnormal{cor}_{I_m}^{I_1}$.  Therefore, it suffices to compute the action of $\omega^{-1}_*$ on $\textnormal{H}^n(I_m,\bbF_p)$.  By Corollary \ref{frattinicor}, $I_m$ is a uniform pro-$p$ group, and \cite[Theorem 5.1.5]{symondsweigel} implies that we have an isomorphism of algebras
$$\bigoplus_{i\geq 0}\textnormal{H}^i(I_m,\bbF_p) \cong \bigoplus_{i \geq 0}\sideset{}{^i_{\bbF_p}}\bigwedge\textnormal{H}^1(I_m,\bbF_p)$$
(the left-hand side equipped with the cup product, the right-hand side with the wedge product).  In particular, $\textnormal{H}^n(I_m,\bbF_p) \cong \bigwedge^n\textnormal{H}^1(I_m,\bbF_p)$, and in order to compute the action of $\omega^{-1}_*$ on $\textnormal{H}^n(I_m,\bbF_p)$ it suffices to compute the determinant of $\omega_*^{-1}$ on $\textnormal{H}^1(I_m,\bbF_p)$.  Finally, by dualizing, it suffices to compute the determinant of conjugation by $\omega$ on the $\bbF_p$-vector space $I_m/I_{m + e}$ (cf. Lemma \ref{frattinilemma}).

Suppose that $t_0\in T_0$.  If $\alpha\in \Phi^+$, then conjugation by $t_0$ acts on $u_\alpha(\fp^{m - 1}/\fp^{m + e - 1})$ as multiplication by $\alpha(t_0)$ on $\fp^{m - 1}/\fp^{m + e - 1}$.  The $\bbF_p$-vector space $\fp^{m - 1}/\fp^{m + e - 1}$ has a filtration
$$0 \subset \fp^{m + e - 2}/\fp^{m - e - 1} \subset \ldots \subset \fp^{m}/\fp^{m + e - 1} \subset \fp^{m - 1}/\fp^{m + e - 1}$$
which is stable by multiplication by $\alpha(t_0)$.  Therefore, the determinant of multiplication by $\alpha(t_0)$ on $\fp^{m - 1}/\fp^{m + e - 1}$ is equal to 
$$\prod_{i = m - 1}^{m + e - 2}\sideset{}{_{\bbF_p}}\det\left(\alpha(t_0)|_{\fp^{i}/\fp^{i + 1}}\right) = \textnormal{N}_{k_F/\bbF_p}\big(\overline{\alpha(t_0)}\big)^{e}.$$ 
The same argument works for $\beta\in \Phi^-$.  Since conjugation by $t_0$ acts trivially on $T_m/T_{m + e}$, we see that
$$\sideset{}{_{\bbF_p}}\det\left(t_0|_{I_m/I_{m + e}}\right) = \prod_{\alpha\in \Phi}\textnormal{N}_{k_F/\bbF_p}\big(\overline{\alpha(t_0)}\big)^{e} = 1,$$
by symmetry of $\Phi$.

Now let $\omega\in \widetilde{\Omega}$ be arbitrary.  Adjusting $\omega$ by an element of $T_0$, we may assume that $\omega\widehat{\bo}^{-1} = \lambda(\varpi)$ for some $\lambda\in X_*(\bT)$.  We choose an $\bbF_p$-basis $\{x_i\}_{i = 1}^{ef}$ for $\fo/p\fo = \fo/\fp^e$.  Let $\alpha\in \Phi$, and suppose that $\{\alpha,\bo(\alpha),\ldots, \bo^{a - 1}(\alpha)\}$ is the orbit of $\alpha$ under $\bo$.  Since $\omega u_{\beta}(\varpi^{m + c_{\beta}}x_i)\omega^{-1} = u_{\bo(\beta)}(\varpi^{m + c_{\bo(\beta)}}d_{\bo,\beta}x_i)$, where $c_\beta := -\mathbf{1}_{\Phi^+}(\beta)$ and $d_{\bo,\beta} = \pm 1$, it follows that the space 
\begin{equation}\label{htopspan}
\textnormal{span}_{\bbF_p}\left\{u_{\alpha}(\varpi^{m + c_\alpha}x_i),~u_{\bo(\alpha)}(\varpi^{m + c_{\bo(\alpha)}}x_i),~\ldots,~u_{\bo^{a - 1}(\alpha)}(\varpi^{m + c_{\bo^{a - 1}(\alpha)}}x_i)\right\}
\end{equation}
is stable by conjugation by $\omega$.  Further, the determinant of $\omega$ on this space is equal to $(-1)^{a - 1}d_{\bo,\alpha}d_{\bo,\bo(\alpha)}\cdots d_{\bo,\bo^{a - 1}(\alpha)}$.  If the orbit of $-\alpha$ under $\bo$ is disjoint from the orbit of $\alpha$, then the determinant of $\omega$ on the direct sum of the spaces \eqref{htopspan} associated to $\alpha$ and $-\alpha$ is equal to
$$(-1)^{a - 1}d_{\bo,\alpha}d_{\bo,\bo(\alpha)}\cdots d_{\bo,\bo^{a - 1}(\alpha)}\cdot (-1)^{a - 1}d_{\bo,-\alpha}d_{\bo,-\bo(\alpha)}\cdots d_{\bo,-\bo^{a - 1}(\alpha)} = 1,$$
since $d_{w,\beta} = d_{w,-\beta}$.  On the other hand, if $-\alpha$ is contained in the orbit of $\alpha$ under $\bo$, then we have $-\alpha = \bo^{a'}(\alpha)$ for some $a'$, and the orbit of $\alpha$ is equal to 
$$\left\{\alpha, \bo(\alpha),\ldots, \bo^{a' - 1}(\alpha),-\alpha,-\bo(\alpha),\ldots, -\bo^{a' - 1}(\alpha) \right\}.$$
Thus, the determinant of $\omega$ on the subspace \eqref{htopspan} is equal to
$$(-1)^{2a' - 1}d_{\bo,\alpha}\cdots d_{\bo,\bo^{a' - 1}(\alpha)}d_{\bo,-\alpha}\cdots d_{\bo,-\bo^{a' - 1}(\alpha)} = -1$$
(again using that $d_{w,\beta} = d_{w,-\beta}$).  Finally, under the isomorphism $T_m/T_{m + e}\cong X_*(\bT)\otimes_{\bbZ}(1 + \fp^m)/(1 + \fp^{m + e})$, the conjugation action of $\omega$ on $T_m/T_{m + e}$ becomes the conjugation action of $\omega$ on $X_*(\bT)$, and therefore the determinant of $\omega$ on $T_m/T_{m + e}$ is $(-1)^{ef\boldl(\bo)}$.  Combining all of these facts, we have
$$\sideset{}{_{\bbF_p}}\det\left(\omega|_{I_m/I_{m + e}}\right) = (-1)^{(N + \boldl(\bo))ef},$$
where $N$ is the number of orbits of $\bo$ in $\Phi$ which are stable by $\alpha\longmapsto -\alpha$.  A straightforward exercise (using, for example, the tables in \cite[Section 1.8]{iwahorimatsumoto}) shows that $N + \boldl(\bo)$ is always even, and therefore
$$\sideset{}{_{\bbF_p}}\det\left(\omega|_{I_m/I_{m + e}}\right) = 1.$$
\end{proof}

Finally, we examine how the $\cH$-modules $\textnormal{H}^i(I_1, C)$ behave under duality.  We continue to suppose that $\Phi$ is irreducible.  For a right $\cH$-module $\fm$, we let $\fm^*$ denote the dual space $\Hom_{C}(\fm, C)$, which we endow with a right action of $\cH$ via $(f\cdot \T_g)(m) = f(m\cdot\T_{g^{-1}})$.  

\begin{thm}\label{poincare}
Suppose $I_1$ is torsion free, and let $n := [F:\bbQ_p]\dim(\bG)$ denote the cohomological dimension of $I_1$.  For $0 \leq i \leq n$, the cup product induces an isomorphism of $\cH$-modules
$$\textnormal{H}^i(I_1, C)^* \cong \textnormal{H}^{n - i}(I_1, C).$$
\end{thm}

\begin{proof}
The cup product gives a perfect pairing
$$\textnormal{H}^i(I_1, C) \otimes_C \textnormal{H}^{n - i}(I_1, C) \longrightarrow \textnormal{H}^n(I_1,C) \cong C;$$
thus, it suffices to prove that 
$$\varphi\cdot\T_w \smile \psi = \varphi\smile \psi\cdot\T_{w^{-1}}$$
where $\varphi\in \textnormal{H}^i(I_1, C), \psi\in \textnormal{H}^{n - i}(I_1, C)$, and $w\in \tW$.  Further, by the braid relations, it suffices to assume that $w = \salpha$ for $\fat{\alpha}\in \Pi_\aff$ or $w = \omega\in \widetilde{\Omega}$.

Let $\fat{\alpha} = (\alpha,\ell)\in \Pi_\aff$, and for $m\geq 1$, define $J_m := I_m \cap \salpha I_m\salpha^{-1}$, a normal subgroup of $J_1$.  Since $\salpha^2$ normalizes $I_m$, $\salpha$ normalizes $J_m$.  By an argument exactly analogous to the proof of Lemma \ref{frattinilemma}, if $m > e$ then the Frattini quotient of $J_m$ is equal to
$$J_m/J_{m + e} = T_{m}/T_{m + e}\oplus  u_{\alpha}(\fp^{\mathbf{1}_{\Phi^-}(\alpha) + m}/\fp^{\mathbf{1}_{\Phi^-}(\alpha) + m + e})\oplus\bigoplus_{\beta\in \Phi \smallsetminus\{\alpha\}}u_{\beta}(\fp^{\mathbf{1}_{\Phi^-}(\beta) + m - 1}/\fp^{\mathbf{1}_{\Phi^-}(\beta) + m + e - 1}).$$
In particular, if $m > e$, then $J_m$ is a uniform pro-$p$ group.

We claim that the conjugation action $\salpha_*$ on $\textnormal{H}^n(J_1,C)$ is trivial.  As in the proof of Theorem \ref{htop}, it suffices to show that the determinant of conjugation by $\salpha$ on the $\bbF_p$-vector space $J_m/J_{m + e}$ is 1 for some $m > e$.  Let $\{x_i\}_{i = 1}^{ef}$ denote an $\bbF_p$-basis for $\fo/p\fo = \fo/\fp^e$.  If $\beta\in \Phi\smallsetminus\{\pm \alpha\}$, then we have $\salpha u_\beta(\varpi^{m + c_{\beta}}x_i)\salpha^{-1} = u_{s_\alpha(\beta)}(d_{\alpha,\beta}\varpi^{m + c_{s_\alpha(\beta)}}x_i)$, where $c_{\beta} := -\mathbf{1}_{\Phi^+}(\beta)$.  Therefore, the space
\begin{equation}\label{spansalpha}
\textnormal{span}_{\bbF_p}\left\{u_{\beta}(\varpi^{m + c_\beta}x_i),~u_{s_{\alpha}(\beta)}(\varpi^{m + c_{s_\alpha(\beta)}}x_i)\right\}
\end{equation}
is stable by conjugation by $\salpha$.  The determinant of conjugation on this space is equal to $-d_{\alpha,\beta}d_{\alpha,s_{\alpha}(\beta)}$ if $s_{\alpha}(\beta)\neq \beta$, and $d_{\alpha,\beta}$ if $s_{\alpha}(\beta) = \beta$.  This implies that the determinant on the direct sum of the spaces \eqref{spansalpha} associated to $\beta$ and $-\beta$ is equal to 
$$(-d_{\alpha,\beta}d_{\alpha,s_{\alpha}(\beta)})(-d_{\alpha,-\beta}d_{\alpha,-s_{\alpha}(\beta)}) = 1\qquad \textnormal{or}\qquad d_{\alpha,\beta}d_{\alpha,-\beta} = 1,$$
according to whether $s_{\alpha}(\beta) \neq \beta$ or $s_{\alpha}(\beta) = \beta$ (recall that $d_{\alpha,\beta} = d_{\alpha,-\beta}$).  On the other hand, we have $\salpha u_{\alpha}(\varpi^{m + c_{\alpha}'}x_i)\salpha^{-1} = u_{-\alpha}(d_{\alpha,\alpha}\varpi^{m + c_{-\alpha}'}x_i)$, where $c_\alpha' = \mathbf{1}_{\Phi^-}(\alpha)$ and $c_{-\alpha}' = \mathbf{1}_{\Phi^-}(-\alpha) - 1$.  Thus, the determinant of conjugation by $\salpha$ on the space
$$\textnormal{span}_{\bbF_p}\left\{u_{\alpha}(\varpi^{m + c_\alpha'}x_i),~u_{-\alpha}(\varpi^{m + c_{-\alpha}'}x_i)\right\}$$
is $-d_{\alpha,\alpha}d_{\alpha,-\alpha} = -1$.  Finally, using the isomorphism $T_m/T_{m + e}\cong X_*(\bT)\otimes_{\bbZ} (1 + \fp^m)/(1 + \fp^{m + e})$, we see that the determinant of conjugation by $\salpha$ on $T_m/T_{m + e}$ is $(-1)^{\boldl(s_{\alpha})ef}$.  Combining everything, we get
$$\sideset{}{_{\bbF_p}}\det\left(\salpha|_{J_m/J_{m + e}}\right) = (-1)^{(\boldl(s_{\alpha}) + 1)ef} = 1.$$

Using the above calculations and the definition of $\T_{\salpha}$, we obtain
\begin{eqnarray*}
\varphi\cdot\T_{\salpha} \smile \psi & = & \textnormal{cor}_{J_1}^{I_1}\circ\salpha_*^{-1}\circ\textnormal{res}_{J_1}^{I_1}(\varphi)\smile\psi\\
& = & \textnormal{cor}_{J_1}^{I_1}\left(\salpha_*^{-1}\circ\textnormal{res}_{J_1}^{I_1}(\varphi)\smile\textnormal{res}_{J_1}^{I_1}(\psi)\right)\\
& = & \textnormal{cor}_{J_1}^{I_1}\left(\salpha_*^{-1}\left(\textnormal{res}_{J_1}^{I_1}(\varphi)\smile\salpha_*\circ\textnormal{res}_{J_1}^{I_1}(\psi)\right)\right)\\
& = & \textnormal{cor}_{J_1}^{I_1}\left(\textnormal{res}_{J_1}^{I_1}(\varphi)\smile\salpha_*\circ\textnormal{res}_{J_1}^{I_1}(\psi)\right)\\
& = & \varphi\smile \textnormal{cor}_{J_1}^{I_1}\circ\salpha_*\circ\textnormal{res}_{J_1}^{I_1}(\psi)\\
& = & \varphi\smile \psi\cdot \T_{\salpha^{-1}}.
\end{eqnarray*}
Now let $\omega\in \widetilde{\Omega}$.  Using Theorem \ref{htop}, we get
\begin{eqnarray*}
\varphi\cdot\T_\omega \smile \psi & = & \omega_*^{-1}(\varphi)\smile \psi\\
 & = & \omega_*^{-1}\left(\varphi\smile \omega_*(\psi)\right)\\
 & = & \varphi\smile \omega_*(\psi)\\
 & = & \varphi\smile \psi\cdot\T_{\omega^{-1}}.\\
\end{eqnarray*}
The result follows.  

\end{proof}

\bibliographystyle{amsalpha}
\bibliography{refs}

\end{document}